\documentclass[graybox, leqno]{svmult}

\usepackage{type1cm}        
\usepackage{makeidx}         
\usepackage{graphicx}        
\usepackage{multicol}        
\usepackage[bottom]{footmisc}

\usepackage{newtxtext}       %
\usepackage[varvw]{newtxmath}       

\makeatletter
\@namedef{subjclassname@2020}{\textup{2020} Mathematics Subject Classification}
\makeatother

\usepackage{lipsum}
\usepackage{amsfonts}
\usepackage{graphicx}
\usepackage{epstopdf}
\usepackage{algorithmic}
\usepackage{calligra}
\usepackage{amsfonts,amsmath,amsthm}
\usepackage{mathtools}
 	\usepackage[colorlinks=true,linkcolor=blue, citecolor = blue]{hyperref}
\usepackage[makeroom]{cancel}
\usepackage{autonum}
\usepackage{hhline}
\usepackage{array}
\usepackage{diagbox}
\usepackage{tcolorbox}
\usepackage{mdframed}
\usepackage{multicol}
\usepackage{graphicx}
\usepackage{subcaption}
\usepackage{moreverb}
\usepackage{bbm}
\usepackage[margin=1.25in]{geometry}
\usepackage{todonotes}
\usepackage{scalerel}
\allowdisplaybreaks
\usepackage{mathrsfs}  
\usepackage{lineno}
\usepackage{todonotes}
\usepackage{tikz}
\usepackage{appendix}
\usepackage{enumitem}
\usepackage{pgfplots}
\usetikzlibrary{arrows.meta}
\usepackage{siunitx}
\usepackage[numbers,sort&compress]{natbib}
\definecolor{mygreen}{HTML}{43a047}
\usepackage{subcaption}
\usepackage{doi}
\usepackage{soul}
\usepackage[misc]{ifsym}
\usepackage{siunitx}
\usepackage{dsfont}
 	\definecolor{darkgreen}{rgb}{0.0, 0.2, 0.13}

\newcommand{\Om}{\Omega}

\newcommand{\pt}{p_t}
\newcommand{\ptt}{p_{tt}}

\newcommand{\ds}{\, \textup{d} s }
\newcommand{\dx}{\, \textup{d} \boldsymbol{x}}

\newcommand{\dG}{\, \textup{d} \Gamma}
\newcommand{\dGs}{\, \textup{d} \Gamma \textup{d}s}

\newcommand{\dxs}{\, \textup{d}\boldsymbol{x}\textup{d}s}

\newcommand{\intt}{\int_0^t}

\newcommand{\intO}{\int_{\Omega}}

\newcommand{\R}{\mathbb{R}} 
 
\newcommand{\Ltwo}{L^2(\Omega)}
\newcommand{\Linf}{L^\infty(\Omega)}

\newcommand{\LinfLinf}{L^\infty(0,T; L^\infty(\Omega))}

\def\Linf{L^\infty(\Omega)}

\newcommand{\Ctr}{C_{\textup{tr}}}

\numberwithin{lemma}{section}
\numberwithin{proposition}{section}
\numberwithin{theorem}{section}
\numberwithin{equation}{section}
\makeatletter
\newcommand{\leqnomode}{\tagsleft@true}
\newcommand{\reqnomode}{\tagsleft@false}
\makeatother

\definecolor{grey}{rgb}{0.5,0.5,0.5}

\definecolor{darkgreen}{rgb}{0,0.5,0}

\def\ball{\mathbb{B}_h}

\def\uzeroh{u_{0,h}}

\def\Xp{\mathcal{X}_p}

\def\intG{\int_{\Gamma}}

\def\Xu{\mathcal{X}_u}

\def\eps{\varepsilon}

\def\LtwoTLtwo{L^2(0,T; \Ltwo)}

\def\LinftLinf{L^\infty(0,t; \Linf)}

\def\csq{c^2}

\def\LtwoG{L^2(\Gamma)}

\def\uzero{u_0}

\def\fp{f_p}
\def\fu{f_u}

\def\LtwotLtwo{L^2(0,t; \Ltwo)}

\newcommand{\Gronwall}{Gr\"onwall}

\def\lhs{\textup{lhs}}
\def\rhs{\textup{rhs}}

\newif\ifshow

\def\uht{\uh_t}

\def\calF{\mathcal{F}}

\def\Ih{\mathcal{I}_h}

\usepackage{stmaryrd}
\newcommand{\jump}[1]{\llbracket #1 \rrbracket}
\newcommand{\avg}[1]{\{ \! \! \{ #1 \} \! \! \} }

\newcommand{\triangles}{\mathcal{T}_{h}}

\newcommand{\Ltwonorm}[2]{\|#1\|_{L^2(#2)}}

\newcommand{\dgnorm}[1]{\|#1\|_{\textup{dG}}}
\newcommand{\dgstarseminorm}[1]{|#1|_{\textup{dG,*}}}

\newcommand{\dgseminorm}[1]{| #1 |_{\textup{dG}}}

\newcommand{\dgseminormell}[2]{|#1|_{\textup{dG,#2}}}

\newcommand{\dgnormell}[2]{\|#1\|_{\textup{dG,#2}}}

\newcommand{\opnorm}[1]{\| #1 \|_{\textup{op}}}

\newcommand{\timeder}{\frac{\textup{d}}{\textup{d}t}}

\newcommand{\invconstant}{C_{\textup{inv}}}
\newcommand{\embconstant}{C_{\textup{emb}}}
\newcommand{\embconstantone}{C_{\textup{emb, 1}}}
\newcommand{\embconstanttwo}{C_{\textup{emb, 2}}}

\newcommand{\bupwind}{b_{h}^{\textup{upw}}}

\newcommand{\asip}{a_{h}^{\textup{sip}}}

\newcommand{\ph}{p^{h}}
\newcommand{\qh}{q^{h}}
\newcommand{\wh}{w^{h}}

\newcommand{\uh}{u^{h}}
\newcommand{\fh}{f^{h}}

\newcommand{\Vh}{V^q_{h}}

\newcommand{\phih}{\phi^{h}}

\def\pht{\ph_t}

\newcommand{\eI}{e^{I,u}}

\newcommand{\eh}{e^{h, u}}
\newcommand{\thh}{t^{*}_{h}}

\def\bfv{\boldsymbol{v}}
\def\bfn{\boldsymbol{n}}

\def\Gin{\Gamma_{\textup{in}}}
\def\Gout{\Gamma_{\textup{out}}}

\def\whone{\wh_1}
\def\whtwo{\wh_2}

\def\Vhqu{V_h^q}
\def\Vhqp{V_h^q}

\def\qp{q}
\def\qu{q}

\def\fhp{\fp^h}
\def\fhu{\fu^h}

\def\barh{\bar{h}}

\def\gA{g_{\textup{abs}}}

\def\gAh{g^h_{\textup{abs}}}

\def\calFh{\mathcal{F}_h}
\def\calFhint{\mathcal{F}^{\textup{int}}_h}
\def\calFhbnd{\mathcal{F}^{\textup{bnd}}_h}

\def\hF{h_F}

\def\nF{\boldsymbol{n}_F}

\def\nablah{\nabla_h}

\def\intF{\int_F}

\def\LtwoF{L^2(F)}

\def\Ahsip{A_h^{\textup{sip}}}

\def\Gh{G^h}

\def\Dmin{D_{\textup{min}}}
\def\Dmax{D_{\textup{max}}}

\def\barOmega{\overline{\Omega}}
\def\overlineOmega{\overline{\Omega}}
\def\overlineK{\overline{K}}

\def\epressure{e^p}

\def\intK{\int_K}

\def\Dzero{D_0}
\def\Done{D_1}

\def\Vhq{V_h^q}

\def\uhin{u^h_{\textup{in}}}
\def\ghin{u^h_{\textup{in}}}

\def\uin{u_{\textup{in}}}
\def\gin{u_{\textup{in}}}

\def\DzeroDone{D_0 D_1}

\def\LtwoGin{L^2(\Gin)}

\def\bfx{\boldsymbol{x}}

\def\nK{\boldsymbol{n}_K}

\def\sumK{ \sum_{K \in \triangles}}
\def\sumFint{ \sum_{F \in \calFhint}}
\def\sumFbnd{ \sum_{F \in \calFhbnd}}

\def\bfvnF {\bfv \cdot \nF}

\def\phihup{ \phih_{\textup{up}}}

\def\LtwoK{L^2(K)}

\def\intGin{\int_{\Gin}}

\def\invhF{\frac{1}{\hF}}

\def\hK{h_K}

\def\VGh{V^q_{\Gamma, h}}
\def\VGinh{V^q_{\Gin, h}}

\def\Npartial{N_{\partial}}

\def\card{\textup{card}}

\def\COmegaT{C(\overlineOmega \times [0,T])}

\def\ehp{e^{h,p}}
\def\ehu{e^{h, u}}
\def\ehpt{e_t^{h,p}}

\def\eIp{e^{I, p}}

\def\deltahp{\delta_h}

\def\eIu{e^{I, u}}

\def\openset{\mathcal{A}_h}

\def\ball{\mathcal{B}}

\newcommand{\tnorm}[1]{\vert\!\vert\!\vert #1 \vert\!\vert\!\vert}

\def\Czero{C_{0}}

\def\dGnorm{\textup{dG}}

\def\Gtop{\Gamma_{\textup{top}}}
\def\Gout{\Gamma_{\textup{out}}}

\usepackage{etoolbox}
\makeatletter
\apptocmd{\maketitle}{%
  \@addtoreset{equation}{section}%
}{}{}
\makeatother

\def\closedball{\overline{\ball}}

\def\Lambdah{\Lambda_h}

\def\deltah{\delta_h}
\makeindex             
\newenvironment{eq}
{\begin{equation}\begin{aligned}}
		{\end{aligned}\end{equation}
}

\begin{document}

\title*{Discontinuous Galerkin approximation of a nonlinear multiphysics problem arising in ultrasound-enhanced drug delivery}
\author{Femke de Wit and Vanja Nikoli\'c \orcidID{0000-1111-2222-3333}} 
\institute{Femke de Wit \at FlandersMake @ KU Leuven, Department of Mechanical Engineering, Celestijnenlaan 300, B-3001, Leuven, Belgium, \email{femke.dewit@kuleuven.be} \and Vanja Nikoli\'c \at Radboud University, Department of Mathematics, Heyendaalseweg 135, 6525 AJ Nijmegen, The Netherlands, \email{vanja.nikolic@ru.nl}}
%
%
\maketitle

\abstract{
 Motivated by simulations of ultrasound-enhanced drug delivery, this work presents the numerical analysis of a mathematical model that captures the influence of ultrasound waves on the diffusivity of the drug. The system under study consists of the Westervelt wave equation, accounting for the nonlinear propagation of ultrasound, coupled to a convection-diffusion equation modeling the drug concentration.  In particular, drug delivery is affected by ultrasound through a pressure-dependent diffusion coefficient. The Westervelt equation is supplemented by linear absorbing boundary conditions as a means of reducing spurious reflections off the boundaries of computational domains. For spatial discretization of this multiphysics system, we employ a discontinuous Galerkin approach on simplicial meshes.  Under suitable assumptions on the exact pressure and the mesh size, we first establish well-posedness, non-degeneracy, and optimal convergence rates in the energy norm for the semi-discrete pressure subproblem. The smallness of the semi-discrete pressure is then used to establish the well-posedness and convergence of the wave--convection-diffusion system under suitable regularity of the exact concentration. Finally, theoretical findings are illustrated through numerical experiments.
}

\section{Introduction} \label{sec:Introduction}
Recent research has shown that, among its many medical uses, ultrasound waves can improve targeted drug delivery in cancer treatments~\cite{mo2012ultrasound}. 
The technique is non-invasive and allows one to fine-tune the penetration depth and duration. Ultrasound has important thermal and mechanical effects \cite{azhdari}. Indeed, ultrasound waves are absorbed by the tissue, raising the temperature. This increase in temperature in turn increases diffusion and dilates blood vessels, improving blood flow. An increase in temperature also modifies the structure of the extracellular matrix (ECM), enhancing fluid flow. Furthermore, the pressure oscillations from the ultrasound can create microbubbles filled with gas that can oscillate in a stable or an unstable manner. When the oscillations are stable, they allow for improved fluid flow around the bubbles, called micro-streaming. Microstreaming can increase drug transport by enhancing the vascular permeability of the tissue. Cavitation also affects the ECM, reducing interstitial fluid pressure and again enhancing blood flow.  \\
\indent This work focuses on simulating the influence of ultrasound waves on drug concentration. In particular, we numerically investigate a mathematical multiphysics model consisting of an ultrasound wave model that captures nonlinear effects in wave propagation coupled to a reaction-diffusion equation with a pressure-dependent diffusion coefficient. This coefficient captures mainly the thermal effects of ultrasound. Although this system represents a simplification of the multiphysics interactions involved in ultrasound-targeted drug delivery, it allows us to gain the missing theoretical insight into such interactions and can thus be seen as a step toward considering more realistic systems. We perform the discretization in space using a discontinuous Galerkin (dG) approach, as this class of methods is known for their good stability and approximation properties as well as flexibility when it comes to handling different geometries; we refer, for example, to the books~\cite{dipietro, riviere2008discontinuous, cohen2017finite, CangianDongGeorgoulisHouston2017} as well as works~\cite{bause2023structure, antonietti2016high, grote2006discontinuous, han2019optimal, cockburn1999some, AnMa2018, houston2002discontinuous}, the review~\cite{antonietti2022mathematical}, and the references provided therein for further details. \\
\indent The modeling and simulation of ultrasound-enhanced drug delivery has been an emerging area of research; see, for example,~\cite{zhan2019towards, hariharan2017model, tian2024numerical, careaga2025westervelt}.  Numerical analysis of different models of ultrasound-enhanced drug delivery has been carried out in~\cite{ferreira2022drug, ferreirathird, azhdari2023drug}. Compared to the existing works, the present multiphysics model accounts for the inherent nonlinearity of ultrasound waves via the Westervelt wave equation. Novel contributions also include employing the non-conforming dG framework and considering absorbing acoustic boundary conditions as a means of reducing spurious reflections from the computational boundaries. The resulting wave--advection-diffusion system is sequentially coupled, meaning that the bulk of the numerical analysis concerns the nonlinear acoustic subproblem with absorbing conditions.  In recent years, significant advances have been made in the numerical analysis of Westervelt-type equations. Our analysis, in particular, builds upon the ideas from~\cite{antonietti2020high, abstractlipschitzpaper}; see also~\cite{careaga2026finite, dorich2024robust}.  The core of the arguments is adapted from the thesis~\cite{thesisFemke}. We refer the readers to, for example,~\cite{dorich2025strong, gomez2025asymptotic} for related results on the numerical analysis of wave models in nonlinear acoustics.
\subsection{Mathematical model: A nonlinear wave--convection-diffusion system}
Let $\Omega \subset \R^d$, $d \in \{1,2,3\}$ be an open and bounded set and let $T>0$ denote the time horizon. To model the propagation of ultrasound waves, including nonlinear effects, we employ the \textit{Westervelt equation}~\cite{westerveltoriginal} with strong damping:
\begin{align}\label{eq: westervelt introduction}
	\left((1+\kappa p)\pt\right)_t- \csq \Delta p - \beta \Delta \pt = \fp \qquad \text{in } \Omega \times (0,T).
\end{align}
In \eqref{eq: westervelt introduction}, $p=p(x,t)$ denotes the acoustic pressure, $c>0$ the speed of sound, $\beta>0$ the strong damping parameter, also known as the sound diffusivity,  $\kappa \in \R$ the nonlinearity coefficient, and $\fp \in \LtwoTLtwo$ the acoustic source function.  To minimize spurious reflections from the computational boundaries, we employ absorbing-type boundary conditions on the whole boundary $\Gamma = \partial \Omega$:
\begin{equation}
		\alpha p_{t} + \csq \nabla p \cdot \bfn + \beta \nabla p_{t} \cdot \bfn = \gA \qquad \text{in } \Gamma \times (0,T),
\end{equation}
where $\alpha>0$, and $\bfn$ denotes the outer normal to $\Gamma$. The first-order Engquist--Majda absorbing boundary conditions (see~\cite{abcsengquist}) are recovered for $\alpha=c$ and $\gA \equiv 0$. We allow for a boundary source function $\gA \in L^2(0,T; \LtwoG)$, which will be useful for verifying empirical convergence rates in numerical experiments.\\
\indent  We model drug concentration $u$ using the following \textit{convection-diffusion equation} with a pressure-dependent diffusion coefficient:
\begin{equation} \label{concentration eq}
	u_{t} + \nabla \cdot (\bfv u) - \nabla \cdot (D(p) \nabla u) = \fu.
\end{equation}
Here $\nabla \cdot (\bfv u)$ models the convection of the drug concentration with a given constant convective velocity $\bfv \in \R^{d}$. 
 The term $\nabla \cdot (D(p) \nabla u)$ models the pressure-dependent diffusion of the drug concentration. We assume the linear dependence of $D$ on the pressure of the form 
\begin{eq}
	 D(p) = D_{0}(1+D_{1}p)
\end{eq}
for some $D_0>0$ and $D_1 \in \R$.  The analysis below can also be extended in a straightforward manner to the setting of $D= D_0(1+D_2(p))$, where $D_2$ depends polynomially on $p$ and $D_2(0)=0$. \\
\indent Furthermore, in \eqref{concentration eq}, $\fu \in \LtwoTLtwo$ denotes the source of the drug concentration. On the inflow part of the boundary
 \begin{equation} \label{def Gin}
 	\Gin = \left \{x \in \partial \Omega: \ \bfv \cdot \bfn(x)< 0 \right \}
 \end{equation}
 we supplement the concentration equation with the condition
 \begin{eq}
 (u \bfv - D(p)\nabla u) \cdot \bfn = \gin \bfv \cdot \bfn, \qquad \text{on } \Gin \times (0,T),
 \end{eq}
\noindent where $\gin \in L^2(0,T; L^2(\Gin))$ is a given boundary source of drug.  On the rest of the boundary
 \begin{equation} \label{def Gout}
	\Gout= \partial \Omega \setminus \Gin = \left \{x \in \partial \Omega: \ \bfv \cdot \bfn(x) \geq 0  \right \}
\end{equation}
we impose homogeneous Neumann (natural) boundary conditions 
\[
-D(p)\nabla u \cdot \bfn =0, \qquad \text{in } \Gout \times (0,T).
\]
 \indent Combining the Westervelt equation with the absorbing boundary conditions and the convection-diffusion equation for the drug concentration, and supplementing with initial data, we obtain: 
\begin{equation}\label{exact multiphysics problem} \tag{P}
	\begin{cases}
	\left(	(1+\kappa p) \pt \right)_t -\csq \Delta p - \beta \Delta p_{t} = f_{p} \qquad & \textup{ in }  \Omega \times (0,T) ,\\[1mm]
		\alpha p_{t} + \csq \nabla p \cdot \bfn + \beta \nabla p_{t} \cdot \bfn = \gA & \textup{ on } \Gamma \times (0,T),\\[2mm]
			u_{t} + \nabla \cdot (\bfv u) - \nabla \cdot (D(p) \nabla u) = f_{u} & \textup{ in } \Omega \times (0,T),\\[1mm]
		(u v - D(p)\nabla u)\cdot \bfn = \gin \bfv \cdot \bfn & \textup{ on } \Gin \times (0,T),\\[1mm]
		- D(p)\nabla u\cdot \bfn = 0 & \textup{ on }  \Gout \times (0,T),\\[2mm]
		(p,p_{t},u) = (p_{0},p_{1},u_{0}) & \textup{ at } \Omega  \times \{0\}.
	\end{cases}
\end{equation}
We assume that the given initial data and source functions are such that the exact problem \eqref{exact multiphysics problem} is well-posed with a sufficiently regular solution. More precisely, we assume that there exists a unique $(p, u) \in \Xp \times \Xu$ which solves \eqref{exact multiphysics problem}, where
\begin{eq}\label{def Xp Xu}
	\Xp =&\, W^{1,\infty}(0,T;H^{q+1}(\Omega)) \cap H^{2}(0,T; H^q(\Omega) \cap L^{\infty}(\Omega)), \qquad q \geq \begin{cases}
		\, 1, \quad \text{if } d \in \{1,2\}, \\
		\, 2, \quad \text{if } d = 3,
	\end{cases} \\
	\Xu =&\, H^{1}(0,T;C^{1}(\overline{\Omega})\cap H^{q+1}(\Omega)), 
\end{eq}
and that there exists $r>0$, such that
\begin{eq} \label{non-degeneracy}
|\kappa| \| p \|_{\COmegaT} < r < 1,
\end{eq}
which implies, in particular, that $1+\kappa p > 0$
and thus that the Westervelt equation does not degenerate. Although this precise well-posedness result does not seem available in the literature, there are closely related results that we expect could be adapted to obtain the needed regularity. In particular, well-posedness of the damped Westervelt equation with (nonlinear) absorbing conditions and $\gA=0$ follows from~\cite{westerveltderivation, simonett2017westerveltabcszeroorder}; in~\cite[Theorem 1.1]{simonett2017westerveltabcszeroorder} the following regularity is obtained:
\begin{eq}
    p \in W^{2,\ell}({0,T}; L^\ell(\Omega)) \cap W^{1, \ell}({0,T}; W^{2,\ell}(\Omega)),\quad  {\textup{for all }\:} \ell >3
\end{eq}
with a zero source term and under suitable domain regularity and compatibility conditions.
A Westervelt--convection-diffusion model that includes non-local acoustic attenuation and temperature effects is studied in~\cite{careaga2025westervelt}. In the well-posedness theory, the non-degeneracy condition \eqref{non-degeneracy} is guaranteed by assuming sufficiently small pressure and temperature data. \\[2mm]
\noindent \textbf{Notation}. We use $\lhs \lesssim \rhs$ and $\lhs \gtrsim \rhs $ to denote $\lhs \leq C \cdot \rhs $ and $\lhs \geq C \cdot  \rhs$, where $C>0$ is a constant that does not depend on the discretization parameter. 
\section{Discontinuous Galerkin semi-discretization of the system} \label{sec: dG discretization} 
Let $\triangles$ be a quasi-uniform and shape-regular simplicial mesh that partitions $\Omega$. We define the mesh size as 
\begin{eq}
	h =  \max_K \hK,
\end{eq}
where $\hK$ denotes the diameter of the element $K \in \triangles$. Let $\calF_h$ denote the set of all faces, with $\calFhint$ being the set of interior faces and $\calFhbnd$ boundary faces, so that
	\[
	\calFh = \calFhint \cup \calFhbnd.
	\]
	 For all $F \in \calFh$, we define the local length scale $\hF$ in dimension $d \geq 2$ to be the diameter of the face $F$. In dimension 1, we set $\hF= \min (h_{K_1}, h_{K_2})$ if $F \in \calFhint$ with $F= \partial K_1 \cap \partial K_2$ for some elements $K_{1}, K_{2} \in \triangles$. We set $\hF=h_K$ if $F \in \calFhbnd$ with $F= \partial K \cap \partial \Omega$.\\
	\indent Consider an interior face $F \in \calFhint$ such that $F=\partial K_{1} \cap \partial K_{2}$. On this face, we define the normal $\bfn_{F}= \bfn_{K_{1}}=- \bfn_{K_{2}}$ and introduce the jump and average operators:
	\begin{equation}
		\begin{split}
			\jump{\phi} = \phi_{\vert K_1} - \phi_{\vert K_2},  \qquad \avg{\phi} = \frac{\phi_{\vert K_1}+\phi_{\vert K_2}}{2}.
		\end{split}
	\end{equation}
	If $F \in \calFhbnd$, we set $\jump{\phi} = \phi$ and   $ \avg{\phi} = \phi$.  \\
	\indent Let $\mathbb{P}^{q}(K)$ be the space of polynomials of degree $q \geq 1$ on $K \in \triangles$. The space of approximate solutions on $\triangles$ is defined as the broken polynomial space
	\begin{eq}
		\Vhq = \left \{\phih \in \Ltwo  \mid   \phih_{\vert K} \in \mathbb{P}^{q}(K),\ \forall K \in \triangles \right \}.
	\end{eq}
    Given a mesh $\triangles$, we define the broken Sobolev space $W^{s,\ell}(\triangles)$ 
for $s \geq 1$ and $1 \leq \ell \leq \infty$ by
\begin{equation}
    W^{s,\ell}(\triangles) = \left\{ \phi \in L^{\ell}(\Omega) \mid \phi_{\vert K} 
    \in W^{s,\ell}(K), \ \forall K \in \triangles \right\}
\end{equation}
with the norm
\begin{equation}
    \|\phi\|_{W^{s,\ell}(\triangles)} = \begin{cases} \left( \displaystyle\sum_{K \in \triangles} 
    \|\phi\|^{\ell}_{W^{s,\ell}(K)} \right)^{1/\ell} & \text{if } \ell < \infty, \\[6pt]
    \displaystyle\sup_{K \in \triangles} \|\phi\|_{W^{s,\infty}(K)} & \text{if } \ell = \infty.
    \end{cases}
\end{equation}
Furthermore, we introduce the broken gradient $\nabla_h: H^1(\triangles) \rightarrow [\Ltwo]^d$ by
\begin{equation}
	(\nablah v)_{\vert K} = \nabla (v_{\vert K}), \quad \forall K \in \triangles
\end{equation}
 for $v \in H^1(\triangles)$; cf.~\cite[Def.~1.2.1]{dipietro}. 
 \par To discretize the $-c^2 \Delta p$ term in the wave equation and the diffusive terms in the pressure and concentration equations, we employ the \emph{Symmetric Interior Penalty discontinuous Galerkin} (SIP-dG) form $\asip: \bigl(H^2(\Omega)+\Vhq\bigr) \times H^1(\triangles) \rightarrow \R$ given by 
 \begin{eq}
 	&\asip(D(\cdot); \phih, \wh) \\
 	= & \begin{multlined}[t] \intO D(\cdot) \nablah \phih \cdot \nablah \wh \dx 
 		- \sum_{F \in \calFhint} \intF \left( \jump{\wh}\avg{D(\cdot) \nablah \wh }\cdot \nF + \jump{\phih}\avg{D(\cdot) \nablah \wh }\cdot \nF \right)  \dG \\ 
 		+ \sum_{F \in \calFhint} \int_{F} \frac{\sigma\eta}{\hF} \jump{\phih} \jump{\wh} \dG,
 	\end{multlined}
 \end{eq}
 where $\eta, \sigma>0$ denote the stabilization parameters. If $D \equiv 1$ (as is the case for the Westervelt equation), we simply write $\asip(\phih, \wh)$ and set $\sigma=1$; this approach to discretizing the Westervelt equation has also been adopted in~\cite{antonietti2020high}.\\
 \indent Following, e.g.,~\cite[Sec.~4.2.3]{riviere2008discontinuous}, the convective term in the concentration equation is discretized with the following \emph{upwind} form $\bupwind: \bigl(H^2(\Omega)+\Vhq\bigr) \times H^1(\triangles) \rightarrow \R$:
 \begin{align}\label{def: bupwind}
 	\bupwind(\bfv;\phih,\wh) = -\intO  \phih \left( \bfv \cdot \nablah \wh \right) \dx + \sum_{F \in \calFhint} \int_{F}  \phih_{\textup{up}}  \bfv \cdot \nF \jump{\wh} \dG + \int_{\Gout} \phih \wh \bfv \cdot \nF \dG.
 \end{align}
 Here $\phih_{\textup{up}}$ denotes the upwind value of a function $\phih \in \Vh$ defined as
 \begin{eq}
 	\phihup = \begin{cases}
 		\phih_{\vert K_1}\ \text{  if  } \bfv \cdot \nF \geq 0, \\
 		\phih_{\vert K_2}\ \text{  if  } \bfv \cdot \nF < 0, \\
 	\end{cases}\quad \forall F= \partial K_1 \cap \partial K_2.
 \end{eq}
 The semidiscrete wave--convection-diffusion problem consists of finding $(\ph,\uh) \in C^{2}([0,T];\Vhq) \times C^{1}([0,T];\Vhq)$, such that
 \begin{align}\label{dG semi-discrete wave-adv diff system} \tag{$P_h$}
 	\begin{cases}
 		\int_{\Omega}((1+\kappa \ph)\pht)_t  \whone \dx +  \asip(\csq \ph+ \beta \pht, \whone) + \int_{\Gamma} \alpha \pht \whone \dG 
 		= \int_{\Omega} \fhp \whone \dx + \int_{\Gamma} \gAh \whone \dG ,\\[3mm]
 		\int_{\Omega} \uht \whtwo \dx + \asip(D(\ph); \uh,\whtwo) + \bupwind(\bfv;\uh,\whtwo)= \int_{\Omega} \fhu \whtwo \dx - \int_{\Gin} \uhin \whtwo \, \bfv \cdot \bfn \dG,\\[3mm]
 		\text{ for all }(\whone,\whtwo) \in \Vhq \times \Vhq \text{ and at all times } t \in [0,T], \text{with} \\[2mm]
 		(\ph(0),\ph_{t}(0), \uh(0)) = (\ph_{0},  \ph_1, \uh_0) \in (\Vh)^3.
 	\end{cases}
 \end{align}
 The functions  $\fhp \in  C([0,T]; \Vhq)$, $\gAh \in C([0,T]; \VGh)$, and $\fhu \in C([0,T]; \Vhq)$, $\ghin \in C([0,T]; \VGinh)$ are assumed to approximate $\fp$, $\gA$, and $\fu$, $\gin$ with the following accuracy:
 \begin{eq} \label{approx properties source terms}
 	&	\|\fp - \fhp\|_{L^{2}(0,T;L^{2}(\Omega))} \lesssim h^{q}, \quad 	\|\gA- \gAh\|_{L^{2}(0,T;L^{2}(\Gamma))} \lesssim h^{q},  \\ & \|\fu - \fhu\|_{L^{2}(0,T;L^{2}(\Omega))} \lesssim h^{q}, \quad \| \gin- \ghin\|_{L^2(0,T; \LtwoGin)} \lesssim h^{q},
 \end{eq}
 where $\VGh$ and $\VGinh$ denote the trace spaces of $\Vh$ on $\Gamma$ and $\Gin$, respectively: 
 \begin{eq}
 	\VGh =&\, \{ \gAh \in \LtwoG:  \gAh = \phih_{\vert \Gamma} \text{  for some  } \phih \in \Vh\}, \\
 	 \VGinh =&\, \{ \ghin \in \LtwoG:  \ghin = \phih_{\vert \Gin} \text{  for some  } \phih \in \Vh\}.
 \end{eq}
 \indent We see that in \eqref{dG semi-discrete wave-adv diff system}, the wave and convection-diffusion equations are coupled sequentially, so we can first analyze the wave subproblem and then employ the obtained result for analyzing the whole system. 
 \subsection{Auxiliary results} 
We next recall some standard results from the dG theory that will be used in the analysis below. 
	\subsubsection{Inverse and trace inequalities}
We begin by recalling discrete inverse and trace inequalities. Below we use the notion of the mesh regularity parameter as introduced in~\cite[Def.~1.38]{dipietro}.
	\begin{lemma}[see {\cite[Lemmas~1.46 and 1.50 and Remark 1.47]{dipietro}}]\label{lem: inverse ineq}
	Let the assumptions made on $\triangles$ in this section hold.	The following inverse estimate holds for $1 \leq \ell,\ell' \leq \infty$ and $\phih \in \Vh$, $K \in \triangles$:
		\begin{eq}	\label{inverse inequality}
			\| \phih \|_{L^{\ell}(K)} \leq \invconstant \hK^{d(1/\ell-1/\ell')}\| \phih \|_{L^{\ell'}(K)},
		\end{eq}
		where the constant $\invconstant$ depends on the mesh regularity parameter, $d$, $q$, $\ell$, and $\ell'$. Further, the following discrete trace inequality holds for $\phih \in \Vhq$ and face $F_K \in \calFh$ of a simplex $K \in \triangles$:
		\begin{eq} \label{face element trace ineq}
			\Ltwonorm{\phih}{F_K} \leq \Ctr h_K^{-1/2} \Ltwonorm{\phih}{K},
		\end{eq}
        	where the constant $\Ctr $ depends on the mesh regularity parameter, $d$, and $q$.
	Moreover,
		\begin{equation}\label{boundary element trace ineq}
				\Ltwonorm{\phih}{\partial K} \leq \Ctr h_K^{-1/2} \Npartial^{-1/2} \Ltwonorm{\phih}{K},
		\end{equation}
		where $\Npartial= \displaystyle \max_{K \in \triangles} \card(F_K) = d+1$ is the maximum number of mesh faces composing the boundary of mesh elements.
	\end{lemma}
	In what follows, we will also need the following continuous trace inequality.
	\begin{lemma}[Continuous trace inequality, see~{\cite[Lemma 1.49]{dipietro}}] \label{lem: cont trace ineq}
	Let the assumptions made on $\triangles$ in this section hold.	Then, for all $v \in H^1(\triangles)$, all $K \in \triangles$, and all $F \in \calFh$,
	\begin{equation}  \label{cont trace ineq}
		\|v\|^2_{L^2(F)} \leq C_{\mathrm{cti}}\left(2\|\nabla v\|_{\LtwoK} + d\, \hK^{-1}\|v\|_{L^2(K)}\right)\|v\|_{L^2(K)},
		\label{eq:cont_trace}
	\end{equation}
			where the constant $C_{\mathrm{cti}}$ depends on the mesh regularity parameter.
	\end{lemma}
	\subsubsection{Discrete embeddings}
For the numerical analysis, we require a discrete counterpart of the embedding $W^{1,\ell}(\Omega) \hookrightarrow L^s(\Omega) $, where 
	\begin{eq} \label{assumptions s}
1 \leq s \leq  \frac{\ell d}{d-\ell} \quad \text{ if } \ 1 \leq \ell <d, \quad  \ell \leq s < \infty \quad \text{ if }\  d= \ell \quad \text{ and }\   s = \infty \quad \text{ if  } \ell>d;
	\end{eq}
	see~\cite[Corollary 9.14]{brezis2011functional}. In particular, we are interested in the case $s=3$. To establish this discrete counterpart, we introduce the dG-norm
	\begin{eq}
		\dgnormell{\phih}{$\ell$} = \left( \sum_{K \in \triangles} \| \nabla \phih \|_{L^{\ell}(K)}^{\ell} + \sum_{F \in \calFh} \frac{1}{\hF^{\ell-1}} \| \jump{\phih}\|_{L^{\ell}(F)}^{\ell}  \right)^{1/\ell}
	\end{eq}
for $1 \leq \ell < \infty$, and the semi-norm
	\begin{eq}
	\dgseminormell{\phih}{$\ell$} = \left( \sum_{K \in \triangles} \| \nabla \phih \|_{L^{\ell}(K)}^{\ell} + \sum_{F \in \calFhint} \frac{1}{\hF^{\ell-1}} \| \jump{\phih}\|_{L^{\ell}(F)}^{\ell}  \right)^{1/\ell}.
\end{eq}
 When $\ell=2$, we simply write $\dgnorm{\cdot}$ and $\dgseminorm{\cdot}$ in place of $\dgnormell{\cdot}{2}$ and $\dgseminormell{\cdot}{2}$, respectively. We first recall the following discrete embedding result.

\begin{lemma}[see~{\cite[Thm 5.3]{dipietro}}]\label{lem: discr embedding full dG norm}
		Under the assumptions on $\triangles$ and $\Vhq$ from Section~\ref{sec: dG discretization},	for  $1 \leq \ell <\infty$ and all $s$ satisfying 
			\begin{eq} \label{assumptions s}
			1 \leq s \leq  \frac{\ell d}{d-\ell} \quad \text{ if } \ 1 \leq \ell <d, \quad  1 \leq s < \infty \quad \text{ if }\  d\leq \ell <\infty, \quad \text{ and }\  1 \leq s \leq \infty \quad \text{ if  } d=1,
		\end{eq}
				there exists a positive constant $\embconstant$ depending on $s$, $\ell$, polynomial degree $q$, and the mesh regularity parameter, such that 
				\begin{eq}
						\| \phih\|_{L^{s}(\Omega)} \leq \embconstantone 	\dgnormell{\phih}{$\ell$}
					\end{eq}
for all $\phih \in \Vh$.          
\end{lemma}
In our analysis setting, with absorbing acoustic conditions, we will need to combine Lemma~\ref{lem: discr embedding full dG norm} with the following bound.
\begin{lemma} \label{lemma: fulld dG norm estimate}
	Let $\ell  \in [1,2)$. Then
			there exists a positive constant $\embconstanttwo$ depending on $s$, $\ell$, polynomial degree $q$, and the mesh regularity parameter, such that 
			\begin{eq}
					\dgnormell{\phih}{$\ell$}^\ell \leq \embconstanttwo (\dgseminorm{\phih}^\ell+{h^{1-\ell}}\|\phih\|^\ell_{\LtwoG})
			\end{eq}
for all $\phih \in \Vh$.			
\end{lemma}
\begin{proof}
We wish to show that 
	\begin{eq}
 &\intO | \nablah \phih |^\ell \dx + \sum_{F \in \calFh} \frac{1}{\hF^{\ell-1}} \| \jump{\phih}\|_{L^{\ell}(F)}^{\ell}  \\
 \lesssim&\, \left( \sum_{K \in \triangles} \| \nabla \phih \|_{L^{2}(K)}^{2} + \sum_{F \in \calFhint} \frac{1}{\hF} \| \jump{\phih}\|_{L^{2}(F)}^{2}  \right)^{\ell/2}+h^{1-\ell}\left( \sum_{F \in \calFhbnd}\| {\phih}\|_{L^{2}(F)}^{2}  \right)^{\ell/2}.
\end{eq}	
By H\"older's inequality using $\frac{\ell}{2}+\frac{2-\ell}{2}=1$, we first infer that
\begin{eq}
	\sumK \intK |\nabla \phih |^\ell \dx \leq&\,  \left(\sumK \intK |\nabla \phih|^2 \dx \right)^{\ell/2} \cdot \left(\sumK \intK 1^{\frac{2}{2-\ell}} \dx \right)^{\frac{2-\ell}{2}}\\
	\lesssim&\,|\Omega|^{\frac{2-\ell}{2}} \left( \sum_{K \in \triangles} \| \nabla \phih \|_{L^{2}(K)}^{2}\right)^{\ell/2}. 
\end{eq}
For the terms involving interior faces, again by H\"older's inequality, we have
\begin{eq}  
	\sumFint \frac{1}{\hF^{\ell-1}} \intF |\jump{\phih}|^\ell \dG = &\,	\sumFint \frac{1}{\hF^{\ell/2-1}}  \frac{1}{\hF^{\ell/2}} \intF |\jump{\phih}|^\ell \dG \\
	 \leq&\, \left( \sumFint \frac{1}{\hF } \intF \jump{\phih}^2 \dG \right)^{\ell/2} \cdot  \left( \sumFint \left(\frac{1}{\hF^{\ell/2-1}} \right)^{\frac{2}{2-\ell}} \intF 1^{\frac{2}{2-\ell}}  \dG \right)^{\frac{2-\ell}{2}}\\
	\lesssim&\,  \left( \sumFint \frac{1}{\hF } \intF \jump{\phih}^2 \dG \right)^{\ell/2} \cdot  \left( \sumFint \hF  |F|\right)^{\frac{2-\ell}{2}}.
\end{eq}
For all $F \in \calFhint$, we pick an element $K \in \triangles$ with a face $F$, so that $\hF \leq \hK$ and the continuous trace inequality \eqref{cont trace ineq} with $v=1$ yields $|F| \lesssim \hK^{-1} |K|$; see, e.g.,~\cite[Lemma 5.1]{di2011mathematical} for similar arguments. Thus, 
\begin{eq} \label{interim estimate}
\sumFint \hF  |F| \lesssim |\Omega|.
\end{eq}
For the terms involving boundary faces, we have  
\begin{eq}  
	\sumFbnd \frac{1}{\hF^{\ell-1}} \intF |\jump{\phih}|^\ell \dG 
	\leq&\, \left( \sumFbnd \intF \jump{\phih}^2 \dG \right)^{\ell/2} \cdot  \left( \sumFbnd \left(\frac{1}{\hF^{\ell-1}} \right)^{\frac{2}{2-\ell}} \intF 1^{\frac{2}{2-\ell}}  \dG \right)^{\frac{2-\ell}{2}}\\
	\lesssim&\,  \left( \sumFbnd \intF \jump{\phih}^2 \dG \right)^{\ell/2} \cdot  \left( \sumFbnd \hF^{\frac{2(1-\ell)}{2-\ell}}  |F|\right)^{\frac{2-\ell}{2}}\\
		\lesssim&\,  h^{1-\ell}  \left( \sumFbnd \intF \jump{\phih}^2 \dG \right)^{\ell/2},
\end{eq}
where in the last line we have used the fact that, thanks to the quasi-uniformity of the mesh,
\begin{eq}
\left( \sumFbnd \hF^{\frac{2(1-\ell)}{2-\ell}}  |F|\right)^{\frac{2-\ell}{2}} \lesssim h^{1-\ell} |\partial \Omega|^{\frac{2-\ell}{2}}.
\end{eq}
 Combining the derived bounds leads to the statement.
\end{proof}		
With $s=3$ and $ \ell= \max\left\{\frac{3d}{d+3}, 1 \right\} \in [1,2)$, the assumptions of both Lemmas~\ref{lem: discr embedding full dG norm} and~\ref{lemma: fulld dG norm estimate} hold. Combining the results of the lemmas leads to the following discrete embedding estimate:
\begin{eq} \label{discrete embedding L3}
	\| \phih \|_{L^{3}(\Omega)}\leq \embconstant ( \dgseminorm{\phih} + h^{\frac{1-\ell}{\ell}} \Ltwonorm{\phih}{\Gamma}), \quad \ell= \max\left\{\frac{3d}{d+3}, 1 \right \}, \quad d \in \{1,2,3\},
\end{eq}
for some $\embconstant>0$, independent of $h$, which we use in the error analysis.

 \subsubsection{Properties of the interpolant} \label{sec: interpolant}
We will conduct the error analysis by splitting the error using the interpolant.  We thus introduce
\begin{equation}
	\left(	\Ih \phi \right)_{ \vert K} = \Ih^K \phi, \quad K \in \triangles,
\end{equation}
where $\Ih^K \phi$ is the local continuous interpolant, which satisfies the following standard estimates.
\begin{lemma}[see {\cite[Lemma 4.4.1, Theorem 4.4.4, Corollary 4.4.7]{brennerscott}}] \label{lem:  local interpolant properties}
	Let the assumptions made on $\triangles$ in this section hold.	For $n-\frac{d}{\ell}>0$, $1< \ell \leq \infty$, and $0 \leq i \leq n \leq q+1$, the following bounds hold: 
	\begin{eq}
		|\phi-\Ih^K \phi|_{W^{i,\ell}(K)} \lesssim h^{n-i}|\phi|_{W^{n,\ell}(K)},
	\end{eq}
	and
	\begin{eq} \label{est interpolant approx Linf}
		\| \phi - \Ih^K \phi \|_{L^{\infty}(K)} \lesssim h^{n-d/2}|\phi|_{H^{n}(K)}.
	\end{eq}
	Furthermore, the following stability result holds: 
	\begin{eq} \label{max stability intp}
		\| {\Ih^{K}} \phi \|_{L^\infty(K)} \lesssim  \| \phi \|_{L^{\infty}(K)}.
	\end{eq}
	The hidden constants are independent of both $K$ and $h$.
\end{lemma}
We will also need the following approximation result in the error analysis:
	\begin{eq} \label{est interpolant dGstar}
		\dgstarseminorm{\phi - \Ih \phi } \lesssim h^{n-1}\|\phi\|_{H^{n}({\Om})}, \quad \phi \in H^n(\Omega), \quad  2 \leq n \leq q+1,
	\end{eq}
    where $\dgstarseminorm{\cdot}$ is defined as
	\begin{eq} \label{def dgstarseminorm}
\dgstarseminorm{\phi} = \left(\dgseminorm{\phi}^{2} + \sum_{K\in \triangles} \hK \Ltwonorm{\nabla \phi \cdot \bfn_K}{\partial K}^{2}\right)^{1/2}, 
	\end{eq}
   which follows by combining Lemma~\ref{lem:  local interpolant properties} with the continuous trace inequality.   
\subsection{Properties of the involved bilinear functionals}	
In the error analysis of \eqref{dG semi-discrete wave-adv diff system}, we will heavily rely on the coercivity and boundedness of the involved bilinear functionals.
\begin{lemma}\label{lemma: coercivity sip} Assume that $D \in W^{1, \infty}(\triangles) \cap L^\infty(\Omega)$ and that there exist $\Dmax$, $\Dmin>0$, independent of $h$, such that
	\begin{eq} \label{assumption D}
	&	0< \Dmin \leq D(x)  \leq \Dmax \quad  \text{for}\quad x \in  \overlineK,  \quad \text{for all}\  K \in \triangles.
	\end{eq}
	If the stabilization parameter satisfies $\sigma \eta > \Ctr^2 (d+1)  \dfrac{ \Dmax^2}{\Dmin} $, then the following estimate holds:
	\begin{eq} \label{coercivity seminorm}
	\asip(D(\cdot); \phih,\phih) \gtrsim  \dgseminorm{\phih}^{2} \quad \forall \phih \in \Vhq.
\end{eq}
\end{lemma}

\begin{proof}
This estimate follows by a simple adaptation of the arguments in~\cite[Lemma 4.12]{dipietro}. Indeed, by  employing the trace inequality \eqref{face element trace ineq}, we first obtain
\begin{eq} \label{est2 coerc}
		\sum_{F \in \calFhint} {\hF} \Ltwonorm{\avg{D(\cdot) \nablah \phih} \cdot \nF}{F}^{2} 
		\leq&\, \sum_{K \in \triangles} {\hK} \Ltwonorm{D(\cdot) \nabla \phih \cdot \nK}{\partial K}^{2}\\
		\leq&\, \sum_{K \in \triangles} {\hK} \|D\|^2_{L^\infty(\partial K)} \Ltwonorm{ \nabla \phih \cdot \nK}{\partial K}^{2}\\
	\leq&\, \sum_{K \in \triangles}  \Ctr^2 \Npartial \|D\|^2_{L^\infty(\partial K)} \Ltwonorm{ \nabla \phih}{K}^{2},
\end{eq}
since $\hF \leq \hK$ for all $F \in \calFhint$.   We thus have
\begin{eq}
\left|\sumFint  \intF \jump{\phih}\avg{D(\cdot) \nablah \phih}\cdot \nF  \dG \right| \leq \Ctr \Npartial^{1/2} \Dmax \|\nablah \phih\|_{\Ltwo}\left(\sumFint   \frac{1}{\hF} \Ltwonorm{\jump{\phih}}{F}^{2}\right)^{1/2}.
\end{eq}
Therefore, 
\begin{eq}
			\asip(D(\cdot); \phih, \phih) \geq&\, \begin{multlined}[t] \Dmin \intO | \nablah \phih |^2\dx 
					+ \sumFint  \intF \frac{\sigma \eta}{\hF} \jump{\phih}^2 \dG\\
					-2 \Ctr \Npartial^{-1/2} \Dmax \|\nablah \phih\|_{\Ltwo}\left(\sumFint   \frac{1}{\hF} \Ltwonorm{\jump{\phih}}{F}^{2}\right)^{1/2}.
				\end{multlined}
\end{eq}
By employing the inequality
\begin{eq}
\Dmin	A^2- 2 \beta_0 AB+ \sigma \eta B^2 \geq \frac{\sigma\eta \Dmin- \beta_0^2}{\Dmin+\sigma \eta}(A^2+B^2)
\end{eq}
for
\begin{eq}
	\sigma \eta \Dmin > \beta_0^2 \eqqcolon \left(\Ctr \Npartial^{1/2}\Dmax\right)^2,
\end{eq}
we obtain \eqref{coercivity seminorm}.

\end{proof}
We further need the following boundedness result.
\begin{lemma} \label{lemma: boundedness asip}
Let the assumptions made on $\triangles$ in Section~\ref{sec: dG discretization} and the assumptions made on the function $D$ in Lemma~\ref{lemma: coercivity sip} be satisfied.	Then the following estimate holds:
\begin{eq} \label{boundedness norm}
\left|	\asip(D(\cdot); \phih, \wh) \right| \lesssim	\dgseminorm{\phih} \dgseminorm{\wh}, \quad \forall \phih, \wh \in \Vhq.
\end{eq}
\end{lemma}
\begin{proof}
	By the Cauchy--Schwarz inequality and the assumed regularity of $D(\cdot)$, we have
\begin{eq}
\left|	\intO D(\cdot) \nablah \phih \cdot \nablah \wh \dx  \right| \leq \Dmax \|\nablah \phih\|_{\Ltwo} \|\nablah \wh\|_{\Ltwo} \lesssim	\dgseminorm{\phih} \dgseminorm{\wh}.
\end{eq}	
Similarly, the Cauchy--Schwarz inequality yields
\begin{eq}
	\left| \sumFint \intF  \frac{\sigma\eta}{\hF} \jump{\phih} \jump{\wh} \dG \right| \lesssim&\,  \left(\sumFint \frac{1}{\hF} \| \jump{\phih}\|^2_{\LtwoF}\right)^{1/2}\left(\sumFint \frac{1}{\hF} \| \jump{\wh}\|^2_{\LtwoF}\right)^{1/2}\\
	\lesssim&\,	\dgseminorm{\phih} \dgseminorm{\wh}.
\end{eq}
The remaining terms within $\asip(D(\cdot); \phih, \wh)$ can be estimated as follows:
\begin{eq} \label{2}
	&\left| 	- \sum_{F \in \calFhint} \intF \left( \jump{\wh}\avg{D(\cdot) \nablah \phih }\cdot \nF + \jump{\phih}\avg{D(\cdot) \nablah \wh }\cdot \nF \right)  \dG \right| \\
	\lesssim&\, \begin{multlined}[t] \sum_{F \in \calFhint} \Bigl( \invhF \|\jump{\wh}\|_{\LtwoF} \hF \|\avg{D(\cdot) \nablah \phih }\cdot \nF \|_{\LtwoF} \\\hspace*{3cm}+  \invhF \|\jump{\phih}\|_{\LtwoF} \hF \|\avg{D(\cdot) \nablah \wh }\cdot \nF \|_{\LtwoF} \Bigr),
		\end{multlined}
\end{eq}
from which by employing \eqref{est2 coerc}, the claim follows.
\end{proof}
\begin{remark} \label{remark}
In the error analysis, we will need to bound terms such as $\asip(p-\Ih p, \wh)$. Since $(p-\Ih p)(t) \notin \Vhq$, we do not have access to the discrete trace inequality used in the proof of Lemma~\ref{lemma: boundedness asip}, therefore we cannot apply \eqref{boundedness norm} directly. For  functions $\phi \in H^{2}(\Omega)+\Vhq$, boundedness holds in the following sense:
	\begin{eq}
		\left|		\asip(D(\cdot); \phi, \wh) \right| \lesssim \dgstarseminorm{\phi} \dgseminorm{\wh}.
	\end{eq}
The claim follows analogously to~\cite[Lemma 4.16]{dipietro}.
\end{remark}
We next establish analogous results for the upwind functional $\bupwind(\bfv; \cdot, \cdot)$ .
\begin{lemma} \label{lemma: identity bupwind}
Let the assumptions made on $\triangles$ and $\Vh$ in this section hold and let $\bfv \in \R^d$. Then $\bupwind(\bfv; \cdot, \cdot)$ satisfies the following identity:
\begin{eq}
	\bupwind(\bfv;\phih,\phih)  =& \begin{multlined}[t] 	\frac12	\sumFint \intF| \bfvnF| \jump{\phih}^2 \dG+ \frac12 \sum_{F \subset \Gin} \intF |\bfvnF| (\phih)^2 \dG \\ +\frac12 \sum_{F \subset \Gout} \intF |\bfvnF| (\phih)^2 \dG
		\end{multlined}
\end{eq}
for all $\phih \in \Vhq$.
\end{lemma}
\begin{proof}
	The proof can be found in~\cite[Sec.~4.2.3]{riviere2008discontinuous}. 
\showfalse  	
\ifshow
	We recall here the main arguments for the convenience of the reader. Using Green's formula and the fact that velocity is divergence-free, we have the rewriting
	\begin{eq}
-\sumK \intK \bfv \phih \cdot \nablah \phih \dx =&\,  -\frac12 \sumK \intK \bfv \cdot \nabla (\phih )^2 \dx\\
= &\, -\frac12\sumK \intK  \nabla \cdot (\bfv \phih )^2 \dx \\ 
= &\, -\frac12 \sumK \int_{\partial K}   \bfv \cdot \nK (\phih)^2 \dG \\
=&\, -\frac12 \sumFint \intF  \bfvnF \jump{\phih}^2 \dG- \frac12 \sumFbnd \intF  \bfvnF (\phih)^2 \dG.
	\end{eq}
Therefore, 
\begin{eq}
	b(\bfv; \phih, \phih)
	 =&\, -\sumK \intK \bfv \phih \cdot \nabla  \phih \dx + \sumFint \int_{F}   \bfvnF \phihup  \jump{\phih} \dG + \sum_{F \subset \Gout} \intF \bfvnF (\phih)^2 \dG \\
	 =&\, \begin{multlined}[t]
	 	 -\frac12 \sumFint \intF  \bfvnF \jump{(\phih)^2 }\dG- \frac12 \sumFbnd \intF  \bfvnF (\phih)^2 \dG \\
	 	 + \sumFint \int_{F}   \bfvnF \phihup  \jump{\phih} \dG + \sum_{F \subset \Gout} \intF \bfvnF (\phih)^2 \dG. 
	 \end{multlined}
\end{eq}	
The first and third term on the right-hand side can be rewritten as
\begin{eq}
&\, \begin{multlined}[t]
	-\frac12 \sumFint \intF  \bfvnF \jump{(\phih)^2} \dG
	+ \sumFint \int_{F}   \bfvnF \phihup  \jump{\phih} \dG 
\end{multlined}\\
=&\, \begin{multlined}[t]
	\sumFint \intF \bfvnF(\phihup \jump{\phih}-\frac12 \jump{(\phih)^2})\dG.
\end{multlined}
\end{eq}
Thus,
\begin{eq}
	&b(\bfv; \phih, \phih)\\
	=&\, \begin{multlined}[t]
		\sumFint \intF \bfvnF(\phihup \jump{\phih}-\frac12 \jump{(\phih)^2})\dG-\frac12 \sumFbnd \intF \bfvnF (\phih)^2 \dG  + \sum_{F \subset \Gout} \intF \bfvnF (\phih)^2 \dG
	\end{multlined}\\
		=&\, \begin{multlined}[t]
		\sumFint \intF \bfvnF(\phihup \jump{\phih}-\avg{\phih}\jump{\phih})\dG-\frac12 \sum_{F \subset \Gin}\intF \bfvnF (\phih)^2 \dG  +\frac12 \sum_{F \subset \Gout} \intF \bfvnF (\phih)^2 \dG
	\end{multlined}\\
			=&\, \begin{multlined}[t]
	\frac12	\sumFint \intF| \bfvnF| \jump{\phih}^2\dG+ \frac12 \sum_{F \subset \Gin} \intF |\bfvnF| (\phih)^2 \dG  +\frac12 \sum_{F \subset \Gout} \intF |\bfvnF| (\phih)^2 \dG,
	\end{multlined}
\end{eq}		
as claimed.
\fi
\end{proof}
The next result concerns the boundedness of the upwind functional.
\begin{lemma} \label{lemma: bound bupwind}
Let the assumptions made on $\triangles$ and $\Vh$ in this section hold.  Then the following bound holds:
\begin{eq}
|\bupwind(\bfv;\phi,\wh)| \lesssim&\, \begin{multlined}[t]
 \gamma \left( \dgseminorm{\wh}^2 + \sumFint  \| |\bfvnF|^{1/2} \jump{\wh }\|^2_{\LtwoF}  + \sum_{F \subset \Gout}  \| |\bfvnF|^{1/2} \jump{\wh }\|^2_{\LtwoF} \right) \\+ \|\phi\|^2_{\Ltwo}+ \sumFint \|\phi_{\textup{up}}\|^2_{L^2(F)} 
 +\sum_{F\subset \Gout}  \|\phi\|^2_{\LtwoF}
\end{multlined}
\end{eq}
for all $\phi \in H^2(\Omega)+\Vh$, $\wh \in \Vhq$, and any $\gamma>0$. 
\end{lemma}
\begin{proof}
The arguments can be found in the proof of~\cite[Theorem 4.2]{riviere2008discontinuous}. 
\showfalse
\ifshow
We include them here for completeness.  By the Cauchy--Schwarz and Young's inequalities, 
\begin{eq}
\left|- \sumK \intK  \phih \left( \bfv \cdot \nabla \wh \right) \dx \right| \lesssim&\, \sumK \|\phih\|_{\LtwoK} \|\nabla \wh\|_{\LtwoK}\\
\lesssim&\, \|\phih\|^2_{\Ltwo}+ \gamma \dgseminorm{\wh}^2
\end{eq}
for any $\gamma>0$. Similarly, 
\begin{eq}
\left|	\sumFint  \intF  \phihup \bfvnF \jump{\wh} \dG \right| 
\lesssim&\,  \gamma \sumFint  \| |\bfvnF|^{1/2} \jump{\wh }\|^2_{\LtwoF} + \sumFint \|\phihup\|^2_{\LtwoF},
\end{eq}
and
\begin{eq}
\left|	\sum_{F\subset \Gout} \intF \bfvnF \phih \wh \dG \right|  \lesssim \gamma \sum_{F\subset \Gout} \| |\bfvnF|^{1/2} \jump{\wh }\|_{\LtwoF}^2+\sum_{F\subset \Gout}  \|\phih\|^2_{L^2(F)},
\end{eq}
for any $\gamma>0$.
Employing these estimates to bound the terms within $\bupwind(\bfv; \phih, \wh)$ leads to the claim.
\fi
\end{proof}
\section{Discontinuous Galerkin analysis of the Westervelt equation with absorbing boundary conditions} 
In this section, we analyze the semi-discrete acoustic subproblem which consists in finding $\ph \in C^{2}([0,T];\Vhq)$, such that 
\begin{align}\label{semidiscrete acoustic subproblem} \tag{$P^p_h$}
	\begin{cases}
		\int_{\Omega}((1+\kappa \ph)\pht)_t  \wh\dx +  \asip(\csq \ph+ \beta \pht,\wh) + \int_{\Gamma} \alpha \pht \wh \dG
		= \int_{\Omega} \fhp \wh \dx + \int_{\Gamma} \gAh \wh \dG,\\[2mm]
		\text{for all }\wh \in \Vhq \text{ and at all times }t \in [0,T],\text{with} \\[1mm]
			(\ph(0),\ph_{t}(0)) = (\ph_{0},  \ph_1). 
	\end{cases}
\end{align}
This result serves as a basis for the numerical analysis of the whole semi-discrete system \eqref{dG semi-discrete wave-adv diff system}.
We next establish existence of the semi-discrete acoustic solution and derive error bounds. The results are stated in the following theorem. 

\begin{theorem}\label{thm: wellposedness nonlin}
Let the acoustic parameters satisfy $\csq$, $ \beta> 0$, and $\kappa \in \R$. Let the assumptions on $\triangles$ made in Section~\ref{sec: dG discretization} hold. Let the polynomial degree be $\qp \geq 1$ for $d \in \{1,2\}$ and $q \geq 2$ for $d=3$. Let $(p_0, p_1) \in H^{q+1}(\Omega) \times C(\Omega)$. Suppose $p \in \Xp$ is the solution of the exact acoustic problem that satisfies the non-degeneracy condition \eqref{non-degeneracy}, where the space $\Xp$ is defined in \eqref{def Xp Xu}.    Let $\fhp$ and $\gAh$ satisfy the accuracy assumptions made in \eqref{approx properties source terms} and let approximate acoustic initial conditions be chosen as
\begin{eq}
	(\ph_{0},  \ph_1) = (\Ih p_{0}, \Ih p_{1}). 
\end{eq}
 Then there exists $\barh>0$, such that for all $h \leq \barh$,  the problem \eqref{semidiscrete acoustic subproblem} has a unique  solution $\ph \in C^2([0,T]; \Vhqp)$ satisfying
\begin{eq} \label{error bound pressure}
	\begin{multlined}[t]
		 \tnorm{p(t)-\ph(t)}^2\coloneqq	\Ltwonorm{p_{t}(t)-\pht(t)}{\Omega}^{2} + \csq \dgseminorm{p(t)-\ph(t)}^{2} + \|p(t)-\ph(t)\|^2_{\LtwoG}\\ \hspace*{2cm}+ \beta \int_{0}^{t} \dgseminorm{\pt(s)-\ph_{t}(s)}^{2} \ds + \int_{0}^{t} \Ltwonorm{\pt(s)-\pht(s)}{\Gamma}^{2} \ds \lesssim h^{2q},
	\end{multlined}
\end{eq}
for all $t \in [0,T]$, where the hidden constant depends on $\|p\|_{\Xp}$ and $T$, but not $h$.
\end{theorem}

To prove this result, we follow the general approach from~\cite{abstractlipschitzpaper} (see also~\cite{vanjawellposedness, careaga2026finite, dorich2025strong}) and adapt it to accommodate the dG setting and absorbing boundary conditions.\\
\indent  In the first step, the semi-discrete acoustic problem \eqref{semidiscrete acoustic subproblem} is rewritten as a system of nonlinear first-order differential equations with a Lipschitz continuous right-hand side. We then show that $(\ph(0),\ph_{t}(0)) \in \openset$, where $\openset$ is the open set defined by
\begin{align}
	\openset = \{ (\phih,\wh)^T \in (\Vhqp)^{2} \mid |\kappa| \| \phih \|_{L^{\infty}(\Omega)} < r+r_{0}<1 \}.
\end{align}
Here $r$ is defined by the non-degeneracy condition of the exact problem \eqref{non-degeneracy} and $r_{0}>0$ is a constant to be determined later. As the right-hand side of the system is Lipschitz continuous on this open domain,
we will be able to apply the local version of the Picard--Lindel\"of theorem~\cite[Problem 7.1.3]{picardlindelofp105} and conclude that there exists a unique solution up to some time $\tilde{t}>0$ that lies in the closure $\overline{\mathcal{A}}_h$. We will then show that there exist $\thh>0$ and $C_{0}>0$, where
\begin{align}\label{eq: def thh}
	\begin{split}
		\thh =  \sup \,  \Bigl \{  t \in (0,T] \mid & \textup{  a unique solution } 
		\ph \in C^{2}([0,t];\Vhqp)\ \textup{of \eqref{semidiscrete acoustic subproblem} exists and }  \\ &  
		\tnorm{\ph(t)-\Ih p(t)} < \Czero h^{d/3+\eps} \Bigr \}
	\end{split}
\end{align}
for $\eps>0$ that will be fixed in the upcoming proof (see \eqref{def eps}). Subsequently, we will establish error bounds on the interval $[0,\thh]$ that do not depend on $\thh$. Assuming $\thh < T$, we will show that $(\ph(\thh),\ph_{t}(\thh)) \in \openset$. By using Lipschitz continuity again, we can conclude that the solution exists beyond $\thh$ and thus $\thh = T$ and the error bounds we established hold on $[0,T]$. \\
\indent We note that having $\beta>0$ in \eqref{error bound pressure} and thus \eqref{eq: def thh} is crucial because it allows us to have a bound on $\int_0^{\thh}|\pht-\Ih p_{t}|^2_{\dG}\ds$, and, in turn, exploit inequality \eqref{discrete embedding L3}. The estimate \eqref{discrete embedding L3} will be an important tool in the forthcoming analysis of the acoustic subproblem.
\subsection{Existence of the semi-discrete acoustic solution on \texorpdfstring{$[0,\thh]$}{x}}

Following the approach outlined above, we begin the proof of Theorem~\ref{thm: wellposedness nonlin} by showing that the approximate initial conditions belong to $\openset$.
\begin{lemma}\label{lem: p0 in D}
	Under the assumptions of Theorem \ref{thm: wellposedness nonlin}, we have $(\ph(0),\ph_{t}(0)) = (\Ih p_0, \Ih p_1) \in \openset$.
\end{lemma}

\begin{proof}
Using \eqref{est interpolant approx Linf}, we can estimate  
\begin{eq}
		\| \kappa\ph(0) \|_{L^{\infty}(\Omega)} = \|\kappa\Ih p_0 \|_{L^{\infty}(\Omega)} 
		&\leq  \| \kappa p(0) \|_{L^{\infty}(\Omega)} +  | \kappa| \| \Ih p_0 - p_0 \|_{L^{\infty}(\Omega)} \\ 
		&\leq r + C h^{\qp+1-d/2}\|p_0\|_{H^{\qp+1}(\Omega)},
\end{eq}
	for some $C>0$, independent of $h$. Choosing $\barh$ small enough, we can find $r_{0}>0$ such that $ \| \kappa \ph(0) \|_{L^{\infty}(\Omega)} < r+r_{0} < 1$. Thus $(\ph(0),\ph_{t}(0)) \in \openset$.
\end{proof}
To write the semi-discrete acoustic problem as a first-order system of ODEs, we introduce the discrete multiplication operator $\Lambdah$ as 
\begin{eq}
 \left(\Lambdah(\ph) \phih, \wh\right)_{\Ltwo} =  \intO (1+\kappa \ph) \phih \wh \dx
\end{eq}
for all $(\phih, \wh) \in  \Vhqp \times \Vh$.  By Lemma \ref{lem: p0 in D}, we have $|\kappa| \| \ph(0)\|_{L^{\infty}(\Omega)} < 1$, which implies that the operator $\Lambda_{h}$ is locally invertible at $t=0$. \\
\indent Next, we define the discrete differential operator $\Ahsip: \Vhqp \rightarrow \Vhqp$, such that 
\begin{eq}
	 \left(\Ahsip \phih, \wh\right)_{\Ltwo} = \asip(\phih,\wh),
\end{eq}
for all $(\phih, \wh) \in  \Vhqp \times \Vh$.  Properties of $\Ahsip$, including coercivity and boundedness, are established in~\cite[Lemma 4.71]{dipietro}. Furthermore, we introduce the extension operator $\Gh: \VGh \rightarrow \Vhq$ for the boundary integral terms: 
\begin{eq}
(\Gh \phih, \wh)_{L^2(\Omega)} = \int_{\Gamma} \phih \wh \dG
\end{eq}
for all $\phih \in \VGh, \; \wh \in \Vh$. 
This allows us to recast the semi-discrete acoustic problem as a first-order system of ODEs:
\begin{align}\label{eq: nonlin system of odes}
	\begin{bmatrix} \ph_{t} \\ \qh_{t}  \end{bmatrix} = \begin{bmatrix} \qh \\ F(\ph, \qh) \end{bmatrix}, \qquad 	\begin{bmatrix} \ph \\ \qh \end{bmatrix}(0) = \begin{bmatrix} \Ih p_0 \\ \Ih p_1 \end{bmatrix},
\end{align}
where the function $F$ is defined as
\begin{eq}
	F(\ph, \qh)= \Lambda_{h}^{-1}(\ph)\left( \fhp - \pi_{h}( {\kappa} (\qh)^{2}) - \csq \Ahsip \ph - \beta \Ahsip \qh -\alpha \Gh( {\qh_{\vert \Gamma}} ) + \Gh \gAh \right),
\end{eq}
and where $\pi_{h}$ is the $L^{2}(\Omega)$ projection onto $\Vhqp$ (needed because $(\qh)^2 \notin \Vh$).   \\
\indent To apply the local version of the Picard--Lindel\"of theorem~\cite[Problem 7.1.3]{picardlindelofp105}  to \eqref{eq: nonlin system of odes}, we introduce a ball centered at the discrete initial data contained in $\openset$. Let the radius $\rho>0$ be small enough so that the ball 
\[
 \ball ((\ph_{0},\ph_{1})^{T}; \rho)=\left \{((\phih, \wh)^T \in (\Vhq)^2:  \|(\phih-\ph_{0}, \wh-\ph_{1})\|_{\Linf} < \rho \right\}
 \]
 satisfies
\[
 \closedball ((\ph_{0},\ph_{1})^{T}; \rho) \subset \openset.
\]
Since $|\kappa| \| \ph \|_{L^{\infty}(\Omega)} < r + r_0 < 1$ for all $(\ph, \qh)^T \in \ball$,  the operator $\Lambdah$ is invertible on $\ball$. Furthermore, we have the following bounds on the operator norm $\opnorm{\cdot}$ of its inverse:
\begin{eq}
		\frac{1}{1+r+r_{0}} \leq \opnorm{\Lambda_{h}^{-1}(\ph)} \leq \frac{1}{1-(r+r_{0})}.
\end{eq} 
The local Lipschitz continuity and boundedness of $F$ in $\closedball$ follows from the fact that $(\Vhqp)^{2}$ is finite-dimensional, and thus all norms are equivalent; see, e.g.,~\cite[Lemma 3.3]{abstractlipschitzpaper}.

\begin{proposition}
	Under the assumptions of Theorem \ref{thm: wellposedness nonlin}, $\thh>0$.
\end{proposition}

\begin{proof}    
	Given that the approximate initial conditions are in $\openset$, $F$ is locally Lipschitz continuous and bounded in $\closedball$, the local version of the Picard--Lindelöf theorem~\cite[Problem 7.1.3]{picardlindelofp105} yields the local existence of a unique solution $\ph \in C^{2}(0,\tilde{t};\Vh)$ for some $\tilde{t}>0$ with $(\ph{(t)},\ph_{t}{(t)}) \in 
    \closedball \subseteq \openset$ for $t \in [0,\tilde{t}]$. \\
    \indent Since 
	$\ph(0)-\Ih p(0)=0$ and $\ph_{t}(0) - \Ih p_{t}(0) = 0$, we have 
    \begin{eq}
    \tnorm{\ph(0)- \Ih p(0)} < C_0 h^{d/3+\eps}.
    \end{eq}
    We can then use the continuity in time of the local solution and its time derivative to conclude that, up to a small, possibly $h$-dependent, time $0<\thh\leq \tilde{t}$, the solution adheres to the conditions in (\ref{eq: def thh}); that is,
      \begin{eq}
    \tnorm{\ph(t)- \Ih p(t)} < C_0 h^{d/3+\eps}, \qquad t \in [0, \thh].
    \end{eq}
    Therefore, $\thh > 0$.
\end{proof}

\subsection{Extending the existence and error bounds to \texorpdfstring{$[0,T]$}{x}}
Our next goal is to establish the $h$-uniform error bounds on $[0,\thh]$. To this end, we split the total approximation error into a discrete 
part $\ehp = \ph - \Ih p$ and an interpolation part $\eIp = \Ih p - p$:
\[
\ph-p = (\ph-\Ih p)+ (\Ih p -p) = \ehp + \eIp.
\]
We introduce a defect $\deltah$ so that $\Ih p$ satisfies
\begin{eq}\label{eq: nonlin projection equation}
		&\begin{multlined}[t]\intO (1+\kappa \ph)\Ih p_{tt} \wh \dx + \intO \kappa \ph_{t} \Ih p_{t} \wh \dx + \csq\asip( \Ih p, \wh) + \beta\asip( \Ih p_{t},\wh) \\+  \int_{\Gamma} \alpha \Ih p_{t} \wh \dG
		 =  \intO \fhp \wh \dx + \intG \gAh \wh \dG+ \intO \delta_{h} \wh \dx,
                 \end{multlined} 
\end{eq}
for all $\wh \in \Vhqp$, $t \in [0,T]$, where the defect $\deltah$ satisfies
\begin{eq} \label{defect 1}
\begin{multlined}[t]	\intO \deltah \wh \dx =  \intO \left( (1+\kappa \ph) \Ih p_{tt} - (1+\kappa p)p_{tt} \right) \wh \dx + \intO \left( \kappa \ph_{t}\Ih p_{t} - \kappa p_{t}^{2}\right) \wh \dx \\
	 + \csq\asip(\Ih p-p,\wh) + \beta\asip(\Ih p_{t} - p_{t},\wh)+ \int_{\Gamma} \alpha (\Ih p_{t} - p_{t})\wh \dG \\ + \intO (\fp-\fhp)\wh \dx + \intG (\gA-\gAh) \wh \dG
\end{multlined}	
\end{eq}
for all $\wh \in \Vhqp$, $t \in [0,T]$. By subtracting \eqref{eq: nonlin projection equation} from the first equation in \eqref{semidiscrete acoustic subproblem}, we see that the discrete pressure error $\ehp$ solves
\begin{align}\label{eq: weak form error}
	\begin{split}
		\intO (1+\kappa \ph)\ehp_{tt} \wh \dx + \intO \kappa \ph_{t} \ehp_{t} \wh \dx + \csq\asip( \ehp, \wh) + \beta\asip( \ehp_{t},\wh) + &\int_{\Gamma} \alpha \ehp_{t} \wh \dG \\
		&=  - \intO \deltahp \wh \dx
	\end{split}
\end{align}
for all $\wh \in \Vh$, $t \in [0,T]$.  To estimate $\ehp$, we require a bound on the defect term $\intO \deltahp \wh \dx $; the following uniform bound on $\pht$ will be useful for this purpose.

\begin{lemma}\label{lem: unif bound pht}
	Let the assumptions of Theorem \ref{thm: wellposedness nonlin} hold. Then $\ph_{t}$ is uniformly bounded on $[0, \thh]$ in the following sense:
	\begin{align} \label{bound pth}
			\|\ph_{t}\|_{L^2(0, \thh; L^{\infty}(\Omega))} \lesssim 1.
	\end{align}
\end{lemma}

\begin{proof} 
By the inverse inequality \eqref{inverse inequality} and the stability of the interpolant in $\Linf$, we have
		\begin{eq}
	\|\pht \|_{L^2(0, \thh; \Linf)} 
		\leq&\, \|\pht-\Ih \pt \|_{L^2(0, \thh; \Linf)} + \|\Ih \pt\|_{L^2(0, \thh; \Linf)} \\
	\lesssim& h^{-d/3} \|\pht-\Ih \pt \|_{L^2(0, \thh; L^{3}(\Omega))}+ \| \pt\|_{L^2(0, \thh; \Linf)}.
\end{eq}
Then by the discrete embedding inequality \eqref{discrete embedding L3}, 
we obtain 
\begin{eq}
	\|\pht\|^2_{L^2(0, \thh; \Linf)}   
	\lesssim& \begin{multlined}[t] h^{-2d/3} \left(\int_0^{\thh} \dgseminorm{\pht-\Ih \pt}^2 \ds +{h^{2\frac{1-\ell}{\ell}}} \int_0^{\thh} \Ltwonorm{\pht - \Ih \pt }{\Gamma}^2 \ds \right)
\\	+\|\pt\|^2_{L^2(0, \thh; \Linf)}.
    \end{multlined}
\end{eq}
By the definition of $\thh$ in \eqref{eq: def thh}, we have $\int_0^{\thh} \dgseminorm{(\pht-\Ih \pt}^2 \ds+\int_0^{\thh} \Ltwonorm{\pht - \Ih \pt }{\Gamma}^2 \ds \leq \Czero^2  h^{2\left(d/3+\eps\right)}$. 
The claim thus follows by choosing 
\begin{eq}\label{def eps}
	\eps = \frac{\ell-1}{\ell},
\end{eq}
where we recall that  $\ell= \max\left\{\frac{3d}{d+3}, 1 \right \}$.
\end{proof}

Next, we estimate the defect term \eqref{defect 1}. 

\begin{lemma}\label{lem: bound defect}
Let the assumptions of Theorem \ref{thm: wellposedness nonlin} hold. Then 
\begin{eq}
\left| \intt \intO \deltahp \wh \dxs \right|  \lesssim&\, \begin{multlined}[t] h^{2\qp}+		\|\ehp\|^2_{L^\infty(0,t; \Ltwo)} +\int_{0}^{t}\Ltwonorm{\ehp_{t}}{\Omega}^{2} \ds\\\hspace*{2cm}+ \gamma \intt \left( \|\wh\|^2_{\Ltwo} +\dgseminorm{\wh}^2 +  \|\wh\|_{\LtwoG}^2 \right)\ds,
		\end{multlined}
\end{eq}
for any $\gamma>0$ and $t \in [0,\thh]$.
\end{lemma}
\begin{proof}
We begin by rewriting \eqref{defect 1} as
\begin{eq}
	\intO \deltahp \wh \dx 
		=&\, \begin{multlined}[t] \intO \left( (1+\kappa \ph) \eIp_{tt} + \kappa(\ehp+\eIp)p_{tt} \right) \wh \dx 
		 + \intO \left( \kappa \ph_{t}\eIp_{t} + \kappa (\ehpt+\eIp_{t})p_{t} \right) \wh \dx \\
		 +\csq\asip( \eIp,\wh) + \beta\asip( \eIp_{t},\wh)+ \int_{\Gamma} \alpha {\eIp_t} \wh \dG +   \intO (\fp-\fhp)\wh \dx \\+   \intG (\gA-\gAh)\wh \dG.
		\end{multlined}
\end{eq}
We then integrate over $(0,t)$ and estimate each term on the right-hand side in turn. Throughout, we apply H\"older's inequality, the interpolation estimates of Lemma~\ref{lem: local interpolant properties}, and Young's inequality $AB \leq \frac{1}{4\gamma}A^2 + \gamma B^2$ for $\gamma > 0$. We first have 
\begin{eq}
	 \intt \intO (1+\kappa \ph) \eIp_{tt}  \wh \dxs 
	  \leq&\, \|1+\kappa \ph\|_{\LinftLinf}\| \eIp_{tt}\|_{\LtwotLtwo}\|\wh\|_{\LtwotLtwo}\\
	  	  \lesssim&\, (1+r+r_0)^2\| \eIp_{tt}\|^2_{\LtwotLtwo}+ \gamma \|\wh\|^2_{\LtwotLtwo}\\
	  	   \lesssim&\,h^{2q}\|\ptt\|^2_{L^2(0,T; H^q(\Omega))}+ \gamma \|\wh\|^2_{\LtwotLtwo} 
\end{eq}
since $\|1+\kappa \ph\|_{\LinftLinf} \lesssim 1+r+r_0$.  Next we estimate
\begin{eq}
\intt&\intO \kappa(\ehp+\eIp)p_{tt}  \wh \dxs \\
	\lesssim&\, \|\ptt\|^2_{L^2(0,T; \Linf)}(\|\ehp\|^2_{L^\infty(0,t; \Ltwo)} +h^{2q} \|p\|^2_{L^\infty(0,T; H^q(\Omega))} )  +\gamma \|\wh\|^2_{\LtwotLtwo}.
\end{eq}
Similarly,
\begin{eq}
\int_0^{t}	\intO  \kappa \ph_{t}\eIp_{t} \wh \dxs  \leq&\,|\kappa| \|\pht\|_{L^2(0,t; \Linf)} \|\eIp_t\|_{L^\infty(0,t; \Ltwo )}\|\wh\|_{L^2(0,t; \Ltwo )} \\
\lesssim&\, \|\eIp_t\|_{L^\infty(0,t; \Ltwo )}^2 + \gamma \|\wh\|^2_{\LtwotLtwo} \\
\lesssim&\,  h^{2q} \|\pt\|^2_{L^{\infty}(0,T; H^q(\Omega))} + \gamma \|\wh\|^2_{\LtwotLtwo}
\end{eq}
since $\|\pht\|_{L^2(0,\thh; \Linf)} \lesssim 1$ by Lemma~\ref{lem: unif bound pht}. We next bound
\begin{eq}
		 &\intt \intO  \kappa (\ehpt+\eIp_{t})p_{t} \wh \dxs \\
		 \lesssim&\, \|\pt\|_{\LinfLinf}^2(\|\ehpt\|^2_{\LtwotLtwo} +h^{2q} \|\pt\|^2_{L^2(0,T; H^q(\Omega))} )  + \gamma \intt \|\wh\|^2_{\Ltwo} \ds.
\end{eq}
By the boundedness of $\asip(\cdot, \cdot)$ discussed in Remark~\ref{remark}, we further have
\begin{eq}
	\intt \left(	\csq\asip( \eIp,\wh) + \beta\asip( \eIp_{t},\wh) \right) \ds \lesssim&\, \intt \dgstarseminorm{\eIp}^2 \ds+\intt \dgstarseminorm{\eIp_t}^2\ds+\gamma \intt \dgseminorm{\wh}^2 \ds \\
	\lesssim&\, h^{2q}+\gamma \intt \dgseminorm{\wh}^2 \ds,
\end{eq}
where we have applied \eqref{est interpolant dGstar} to both $\eIp$ and $\eIp_t$.
Next, by the continuous trace inequality in Lemma~\ref{lem: cont trace ineq} and the discrete trace inequality \eqref{face element trace ineq},
\begin{eq}
\intt	\int_{\Gamma} \alpha {\eIp_t} \wh \dGs =  &\,  \sumFbnd \intt	\intF  \alpha \eIp_t \wh \dGs \\
\lesssim&\,  \sumFbnd \intt \|\eIp_t\|_{\LtwoF}  \|\wh\|_{\LtwoF} \ds \\
\lesssim&\, \sumFbnd  \intt \left(\|\nabla \eIp_t\|_{\LtwoK}+\hK^{-1}\| \eIp_t\|_{\LtwoK} \right)^{1/2} \| \eIp_t\|^{1/2}_{\LtwoK}\, \hK^{-1/2} \|\wh\|_{\LtwoK} \ds.
\end{eq}
Then, again by the approximation properties of the interpolant, we conclude that 
\begin{eq}
	\intt	\int_{\Gamma} \alpha \eIp_t \wh \dGs
	\lesssim&\, \sumFbnd  \intt \hK^{q/2}\hK^{(q+1)/2} \|\pt\|_{H^{q+1}(K)}\, \hK^{-1/2} \|\wh\|_{\LtwoK} \ds \\
	\lesssim&\, h^{2q}+ \gamma \|\wh\|_{L^2(0,t; \Ltwo)}^2.
\end{eq}
Finally, using the assumed accuracy of $\fhp$ and $\gAh$ in \eqref{approx properties source terms},
\begin{eq}
	&\intt \intO (\fp-\fhp)\wh \dxs+   \intt\intG (\gA-\gAh)\wh \dGs \\
	\lesssim&\, h^{2q}+ \gamma \intt \left( \|\wh\|_{\Ltwo}^{2}+\|\wh\|_{\LtwoG}^{2}\right) \ds.
\end{eq}
Combining all of the above estimates yields the result.
	
\end{proof}
Equipped with Lemma~\ref{lem: bound defect}, we now derive a uniform estimate of the discrete acoustic error.
\begin{proposition}\label{prop: nonlin error bound}
	Under the assumptions of Theorem~\ref{thm: wellposedness nonlin}, $\ehp=\ph - \Ih p$ satisfies
\begin{eq}
	\begin{multlined}[t]	\max_{t \in [0, \thh]}\Ltwonorm{\ehpt(t)}{\Omega}^{2} + \max_{t \in [0, \thh]}\dgseminorm{\ehp(t)}^{2} +\max_{t \in [0, \thh]}\Ltwonorm{\ehp(t)}{\Gamma}^{2} + \int_{0}^{\thh} \dgseminorm{\ehpt}^{2} \ds \\+ \int_{0}^{\thh} \Ltwonorm{\ehp_{t}}{\Gamma}^{2} \ds \lesssim h^{2\qp},
   \end{multlined} 
\end{eq}
where
	the hidden constant depends on $\|p\|_{\Xp}$ and $T$, but not on $h$.
\end{proposition}
\begin{proof}
Choosing $\wh=\ehp_{t}(t) \in \Vhqp$ in \eqref{eq: nonlin projection equation} yields
\begin{eq}
&	\begin{multlined}[t]	\frac{1}{2}\timeder \intO (1+\kappa \ph)(\ehp_{t})^{2} \dx 
		+ \frac{1}{2}\timeder \csq\asip( \ehp,\ehp) +\beta \asip( \ehp_{t},\ehp_{t}) + \int_{\Gamma}  \alpha  (\ehp_{t})^{2}	 \dG
\end{multlined}\\
		=&\,  - \intO \kappa \ph_{t} (\ehp_{t})^{2} \dx +\frac{1}{2} \intO \kappa \ph_{t} (\ehp_{t})^{2} \dx 
		- \intO \deltahp \ehp_{t} \dx. 
\end{eq}
Integrating over $(0,t)$, making use of the coercivity of $\asip(\cdot, \cdot)$ established in Lemma~\ref{lemma: coercivity sip}, the bound $1+\kappa \ph \geq 1-(r+r_{0})>0$ (since $\ph(t) \in {\openset}$), and the fact that the approximate initial conditions interpolate the exact ones yields
\begin{eq}\label{eq: nonlin error first form}
		&\Ltwonorm{\ehp_{t}(t)}{\Omega}^{2} + \csq \dgseminorm{\ehp(t)}^{2} + \beta \intt \dgseminorm{\ehp_{t}(s)}^{2} \ds + \alpha \intt \Ltwonorm{\ehp_{t}(s)}{\Gamma}^{2} \ds \\
			\lesssim&\, 
			\int_{0}^{t} \intO |\kappa \ph_{t}| (\ehp_{t})^{2}  \dx \ds +
			\left|\int_{0}^{t}\intO  \deltahp \ehpt \dxs \right|.
\end{eq}
To estimate the first term on the right-hand side, we use  H\"older's inequality in space and Young's inequality:
	\begin{eq}
	    \int_{0}^{t} \intO | \kappa \ph_{t}| (\ehp_{t})^{2}  \dxs 
        \lesssim&\, \intt \left(\|\ehp_{t}\|^2_{\Ltwo}+\gamma \|\ph_{t}\|_{\Linf}^2\|\ehp_{t}\|^2_{\Ltwo} \right)\ds \\
        \lesssim&\, \|\ehp_{t}\|^2_{L^{2}(0,t;L^{2}(\Omega)} + \gamma \max_{0 \leq s \leq t}\|\ehp_{t}(s)\|^2_{\Ltwo}
	\end{eq}   
for any $\gamma>0$, where in the last step we have applied Lemma~\ref{lem: unif bound pht} to bound $\|\ph_{t}\|_{L^2(0, \thh; \Linf)} \lesssim 1$. The second term on the right-hand side of \eqref{eq: nonlin error first form} can be estimated using Lemma~\ref{lem: bound defect}:
\begin{eq}
\left| \intt \deltahp \ehpt \dxs \right|  \lesssim&\, \begin{multlined}[t] h^{2\qp}+  \|\ehp\|^2_{L^\infty(0,t; \Ltwo)} \\+  \intt \|\ehpt(s)\|^2_{\Ltwo} \ds+\gamma \intt \dgseminorm{\ehpt(s)}^2 \ds+ \gamma \intt \|\ehpt(s)\|^2_{\LtwoG}\ds,
\end{multlined}
\end{eq}
where we can further bound
\begin{eq}
    \|\ehp\|_{L^\infty(0,t; \Ltwo)} \lesssim \sqrt{T}  \|\ehp_t\|_{L^2(0,t; \Ltwo)}.
\end{eq}
In this manner, from \eqref{eq: nonlin error first form}, we arrive at 
	\begin{eq} \label{interim}
		& \Ltwonorm{\ehp_{t}(t)}{\Omega}^{2} + \dgseminorm{\ehp(t)}^{2} + \int_{0}^{t} \dgseminorm{\ehp_{t}}^{2} \ds + \int_{0}^{t} \Ltwonorm{\ehp_{t}}{\Gamma}^{2}\ds \\
        \lesssim&\, \begin{multlined}[t] 
		h^{2\qp} +\int_{0}^{t} \Ltwonorm{\ehp_{t}(s)}{\Omega}^{2} \ds +\gamma \intt \dgseminorm{\ehpt(s)}^2 \ds+ \gamma \intt \|\ehpt(s)\|^2_{\LtwoG}\ds \\
      + \gamma \max_{0 \leq s \leq t}\|\ehp_{t}(s)\|^2_{\Ltwo}
        \end{multlined}
	\end{eq}
for all $t \in [0, \thh]$. Taking the maximum of \eqref{interim} over all $t \in [0, \tau]$ for $\tau \in [0, \thh]$ and then reducing $\gamma$ so that the $\gamma$ terms can be absorbed by the left-hand side yields 
\begin{eq}
		& \Ltwonorm{\ehp_{t}(\tau)}{\Omega}^{2} + \dgseminorm{\ehp(\tau)}^{2}+ \int_{0}^{\tau} \dgseminorm{\ehp_{t}}^{2} \ds + \int_{0}^{\tau} \Ltwonorm{\ehp_{t}}{\Gamma}^{2}\ds \\
        \lesssim&\, \begin{multlined}[t] 
		h^{2\qp} +\int_{0}^{\tau} \Ltwonorm{\ehp_{t}(s)}{\Omega}^{2} \ds.
        \end{multlined}
	\end{eq}
An application of  \Gronwall's inequality together with $\|\ehp\|_{L^\infty(0,t; \LtwoG)} \lesssim \sqrt{T} \|\ehp_t\|_{L^2(0,t; \LtwoG)}$ completes the proof.
\end{proof}

With Proposition \ref{prop: nonlin error bound}, we now have the most important ingredient to prove Theorem~\ref{thm: wellposedness nonlin}.

\begin{proof}[Proof of Theorem \ref{thm: wellposedness nonlin}]
We can show that $(\ph(\thh),\ph_{t}(\thh)) \in \openset$. Indeed, by involving the interpolant, we obtain
   \begin{eq}\label{eq: wellposedness estimate}
			\|\kappa \ph(\thh) \|_{L^{\infty}(\Omega)} 
			\lesssim&  \|\kappa (\ph-\Ih p)(\thh) \|_{L^{\infty}(\Omega)} + \|\kappa \Ih p(\thh) \|_{L^{\infty}(\Omega)} \\
			\lesssim&\, |\kappa| h^{-d/3}\| (\ph-\Ih p)(\thh) \|_{L^{3}(\Omega)} + |\kappa|\|p(\thh)\|_{\Linf}.
	\end{eq}
	From here, a use of the discrete embedding \eqref{discrete embedding L3} leads to
		\begin{eq}\label{eq: wellposedness estimate1}
			\|\kappa \ph(\thh) \|_{L^{\infty}(\Omega)}
			\lesssim& h^{-d/3} \left(\dgseminorm{(\ph-\Ih p)(\thh)} +h^{\frac{1-\ell}{\ell}} \Ltwonorm{(\ph - \Ih p)(\thh)}{\Gamma}\right)+r \\
						\lesssim& {h^{-d/3-\frac{\ell-1}{\ell}}} \left( \dgseminorm{(\ph-\Ih p)(\thh)} +  \sqrt{T} \|(\ph-\Ih p)_t\|_{L^2(0, \thh; \LtwoG)}\right) +r.
	\end{eq}
Using the error estimate from Proposition \ref{prop: nonlin error bound} then yields 
			\begin{eq}\label{eq: wellposedness estimate2}
			\|\kappa \ph(\thh) \|_{L^{\infty}(\Omega)} 
			\lesssim&  (1+\sqrt{T}) {h^{-d/3-\frac{\ell-1}{\ell}+\qp}}	+ r.
	\end{eq}
By assumption, $q \geq 1$ for $d \in \{1,2\}$ and $q \geq 2$ for $d=3$. We thus have $q+1>d/2$ and, since $\ell= \max\{\frac{3d}{d+3}, 1\}$,
	\begin{eq} \label{bound q}
	q > \frac{d}{3}+ \frac{\ell-1}{\ell} = \begin{cases}
		1/3, \quad &d=1,\\
		2/3+\frac16=\frac56, \quad &d=2, \\
		1+\frac13=\frac43, \quad &d=3.
	\end{cases}
	\end{eq}
	 Therefore,  for small enough $\barh$, we obtain
			\begin{eq}\label{eq: wellposedness estimate3}
		\|\kappa \ph(\thh) \|_{L^{\infty}(\Omega)} 
		< r_{0} + r < 1
\end{eq}
and we conclude that $(\ph(\thh),\ph_{t}(\thh)) \in \openset$. Furthermore,
 \[
 \tnorm{\ph(t) - \Ih p(t)}\lesssim h^{\qp} < \Czero {h^{d/3+\eps}} \quad \text{for  }\ t \in [0, \thh]
 \]
because $d/3+\eps<q$ for $q \geq 1$ if $d \in \{1,2\}$ and $q \geq 2$ if $d=3$. \\
\indent Therefore, if $\thh <T$, we can apply the local version of the Picard--Lindelöf theorem again at $t=\thh$ to conclude that the solution exists beyond $\thh$. Thus we must have $\thh = T$ and the error bound in Proposition \ref{prop: nonlin error bound} holds on $[0,T]$.\\
\indent Since $\ph - p =  \ehp + \Ih p - p$, we can use the approximation properties of the interpolant to conclude that the same error estimate that holds for $\ehp=\ph - \Ih p$ also holds for $ \ph - p$. This concludes the proof of Theorem~\ref{thm: wellposedness nonlin}.
\end{proof}\vspace*{-9mm}
\subsection{Empirical order of convergence for the acoustic pressure}\label{sec: pressure numerics}
Next, we numerically verify the convergence rate indicated by Theorem~\ref{thm: wellposedness nonlin}.
We set $\barOmega = [0,1] \times [0,2]$  and  choose the acoustic data $(\fp, \gA, p_0, p_1)$ so that the exact solution is given by
\begin{equation} \label{exact pressure}
	p = \cos(t)\sin(\pi x)\sin(\frac{\pi}{2} y),
\end{equation}
with the acoustic parameters $\alpha=c$, $\kappa =0.1$, $c=1$, $\beta=0.1$, and the final time $T=0.5$.  The time discretization is performed using the Newmark method with parameters $(0.25, 0.5)$ and is realized through a predictor-corrector approach, with a fixed-point iteration with prescribed tolerance to tackle nonlinearities, following~\cite[Ch.~5]{kaltenbacher}.  The simulations are performed using FEniCSx~\cite{fenicsx}.\footnote{The program code for the simulations in this work is available as an ancillary file from the arXiv page~\cite{dewit2026}.} \\
\indent In numerical experiments, we choose the time step $\Delta t = \mathcal{O}(h^{\qp+1})$ to ensure that temporal discretization errors do not pollute the spatial convergence rates. Figure~\ref{fig:westervelt abcs error} shows the plots for different discrete errors using $\qp \in \{1,2\}$. We observe that the first error achieves the order expected by our theory, whereas the order of the errors measured in the $L^\infty(0,T; \Ltwo)$ norm is one order higher, as is known to hold for dG methods; see, e.g.,~\cite[Corollary 4.26]{riviere2008discontinuous}.
\begin{figure}[h]
	\centering
	\begin{subfigure}[b]{0.49\textwidth}
		\centering
	\resizebox{\textwidth}{!}{
\begin{tikzpicture}

\definecolor{darkgray176}{RGB}{176,176,176}
\definecolor{darkorange25512714}{RGB}{255,127,14}
\definecolor{lightgray204}{RGB}{204,204,204}
\definecolor{steelblue31119180}{RGB}{31,119,180}

\begin{axis}[
legend cell align={left},
legend style={
  fill opacity=0.8,
  draw opacity=1,
  text opacity=1,
  at={(0.67,0.4)},
  anchor=south east,
  draw=lightgray204
},
log basis x={10},
log basis y={10},
minor xtick={0.002,0.003,0.004,0.005,0.006,0.007,0.008,0.009,0.02,0.03,0.04,0.05,0.06,0.07,0.08,0.09,0.2,0.3,0.4,0.5,0.6,0.7,0.8,0.9,2,3,4,5,6,7,8,9,20,30,40,50,60,70,80,90},
%
tick align=outside,
tick pos=left,
title={Polynomial degree $q = 1$},
x grid style={darkgray176},
xlabel={\(\displaystyle h\)},
xmin=0.0675441905923626, xmax=0.185064028310577,
xmode=log,
xtick style={color=black},
xtick={0.001,0.01,0.1,1,10},
xticklabels={
  \(\displaystyle {10^{-3}}\),
  \(\displaystyle {10^{-2}}\),
  \(\displaystyle {10^{-1}}\),
  \(\displaystyle {10^{0}}\),
  \(\displaystyle {10^{1}}\)
},
y grid style={darkgray176},
ylabel={discrete error -- pressure},
ymin=0.00405838385940591, ymax=0.168054674153944,
ymode=log,
ytick style={color=black},
ytick={0.0001,0.001,0.01,0.1,1,10},
yticklabels={
  \(\displaystyle {10^{-4}}\),
  \(\displaystyle {10^{-3}}\),
  \(\displaystyle {10^{-2}}\),
  \(\displaystyle {10^{-1}}\),
  \(\displaystyle {10^{0}}\),
  \(\displaystyle {10^{1}}\)
}
]
\addplot [semithick, steelblue31119180, mark=*, mark size=3, mark options={solid}]
table {%
0.176776695296637 0.1418881628307
0.117851130197758 0.0903614996228256
0.0883883476483184 0.0661848614004209
0.0707106781186548 0.0522107031228204
};
\addlegendentry{$\max_{t \in [0,T]} |e^{h,p}(t)|_{\dGnorm}$}
\addplot [semithick, darkorange25512714, mark=*, mark size=3, mark options={solid}]
table {%
0.176776695296637 0.0280242575312351
0.117851130197758 0.0131226623924627
0.0883883476483184 0.00748200657725841
0.0707106781186548 0.0048068166045527
};
\addlegendentry{$\max_{t \in [0,T]} |e^{h,p}_t(t)|_{L^2(\Omega)}$}
\draw (axis cs:0.144337567297406,0.0792616005032307) node[
  scale=0.5,
  text=black,
  rotate=0.0
]{\Large \bf 1.1128};
\draw (axis cs:0.102062072615966,0.0541338713833653) node[
  scale=0.5,
  text=black,
  rotate=0.0
]{\Large \bf 1.0823};
\draw (axis cs:0.0790569415042095,0.0431487969860988) node[
  scale=0.5,
  text=black,
  rotate=0.0
]{\Large \bf 1.0628};
\draw (axis cs:0.144337567297406,0.0134238186253799) node[
  scale=0.5,
  text=black,
  rotate=0.0
]{\Large \bf 1.8713};
\draw (axis cs:0.102062072615966,0.00693614335942233) node[
  scale=0.5,
  text=black,
  rotate=0.0
]{\Large \bf 1.9530};
\draw (axis cs:0.0790569415042095,0.0048979364443688) node[
  scale=0.5,
  text=black,
  rotate=0.0
]{\Large \bf 1.9829};
\end{axis}

\end{tikzpicture}}
	\end{subfigure}
	\begin{subfigure}[b]{0.49\textwidth}
		\centering
		\resizebox{\textwidth}{!}{
\begin{tikzpicture}

\definecolor{darkgray176}{RGB}{176,176,176}
\definecolor{darkorange25512714}{RGB}{255,127,14}
\definecolor{lightgray204}{RGB}{204,204,204}
\definecolor{steelblue31119180}{RGB}{31,119,180}

\begin{axis}[
legend cell align={left},
legend style={
  fill opacity=0.8,
  draw opacity=1,
  text opacity=1,
  at={(0.03,0.97)},
  anchor=north west,
  draw=lightgray204
},
log basis x={10},
log basis y={10},
minor xtick={0.002,0.003,0.004,0.005,0.006,0.007,0.008,0.009,0.02,0.03,0.04,0.05,0.06,0.07,0.08,0.09,0.2,0.3,0.4,0.5,0.6,0.7,0.8,0.9,2,3,4,5,6,7,8,9,20,30,40,50,60,70,80,90},
%
tick align=outside,
tick pos=left,
title={Polynomial degree $q = 2$},
x grid style={darkgray176},
xlabel={\(\displaystyle h\)},
xmin=0.0675441905923626, xmax=0.185064028310577,
xmode=log,
xtick style={color=black},
xtick={0.001,0.01,0.1,1,10},
xticklabels={
  \(\displaystyle {10^{-3}}\),
  \(\displaystyle {10^{-2}}\),
  \(\displaystyle {10^{-1}}\),
  \(\displaystyle {10^{0}}\),
  \(\displaystyle {10^{1}}\)
},
y grid style={darkgray176},
ylabel={discrete error -- pressure},
ymin=8.13294372445965e-05, ymax=0.00726114562033014,
ymode=log,
ytick style={color=black},
ytick={1e-06,1e-05,0.0001,0.001,0.01,0.1},
yticklabels={
  \(\displaystyle {10^{-6}}\),
  \(\displaystyle {10^{-5}}\),
  \(\displaystyle {10^{-4}}\),
  \(\displaystyle {10^{-3}}\),
  \(\displaystyle {10^{-2}}\),
  \(\displaystyle {10^{-1}}\)
}
]
\addplot [semithick, steelblue31119180, mark=*, mark size=3, mark options={solid}]
table {%
0.176776695296637 0.0059201724359692
0.117851130197758 0.0024540957628792
0.0883883476483184 0.00132730380567255
0.0707106781186548 0.000828244127145295
};
\addlegendentry{$\max_{t \in [0,T]} |e^{h,p}(t)|_{\dGnorm}$}
\addplot [semithick, darkorange25512714, mark=*, mark size=3, mark options={solid}]
table {%
0.176776695296637 0.00156151287306523
0.117851130197758 0.000459907458408615
0.0883883476483184 0.000194667640630837
0.0707106781186548 9.97512983683621e-05
};
\addlegendentry{$\max_{t \in [0,T]} |e^{h,p}_t(t)|_{L^2(\Omega)}$}
\draw (axis cs:0.144337567297406,0.00266815448286018) node[
  scale=0.5,
  text=black,
  rotate=0.0
]{\Large \bf 2.1718};
\draw (axis cs:0.102062072615966,0.00126336535345943) node[
  scale=0.5,
  text=black,
  rotate=0.0
]{\Large \bf 2.1364};
\draw (axis cs:0.0790569415042095,0.000733943100773543) node[
  scale=0.5,
  text=black,
  rotate=0.0
]{\Large \bf 2.1134};
\draw (axis cs:0.144337567297406,0.000593206704441751) node[
  scale=0.5,
  text=black,
  rotate=0.0
]{\Large \bf 3.0148};
\draw (axis cs:0.102062072615966,0.000209449895965827) node[
  scale=0.5,
  text=black,
  rotate=0.0
]{\Large \bf 2.9885};
\draw (axis cs:0.0790569415042095,9.75448176613374e-05) node[
  scale=0.5,
  text=black,
  rotate=0.0
]{\Large \bf 2.9963};
\end{axis}

\end{tikzpicture}}
	\end{subfigure}
	\caption{Discrete errors for the pressure}
	\label{fig:westervelt abcs error}
\end{figure}
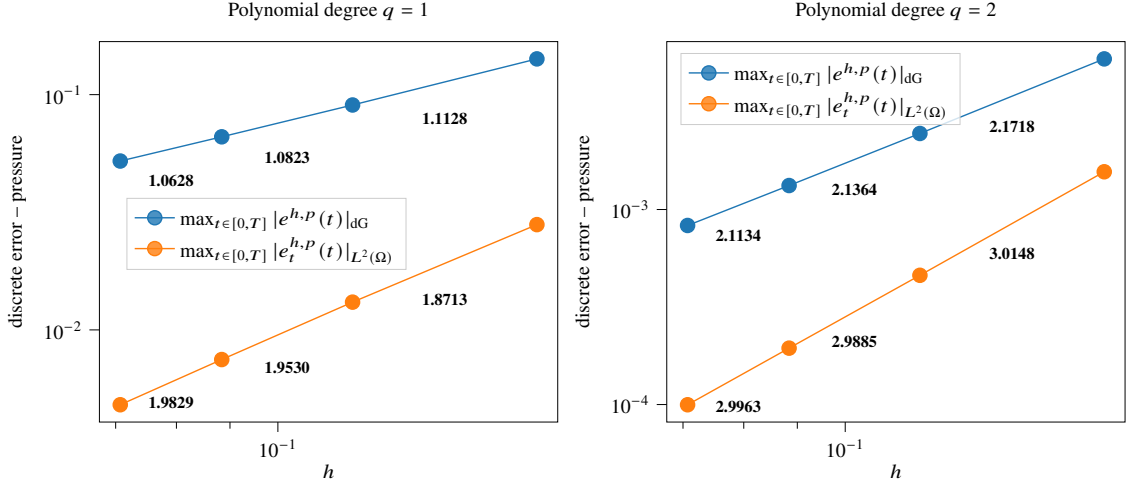

\newpage
\section{Discontinuous Galerkin analysis of the wave--convection-diffusion system}\label{sec: improved model}
Having analyzed the pressure subproblem, we now use these results to establish the error in the approximate concentration. Theorem~\ref{thm: wellposedness nonlin} guarantees the existence of a unique $\ph \in C^2([0,T]; \Vhqp)$ that solves \eqref{semidiscrete acoustic subproblem}, and furthermore (see \eqref{eq: wellposedness estimate2})
\begin{eq} \label{est ph conc new}
	\| \ph \|_{\Linf}  \lesssim&\, \| p \|_{\Linf} + \barh^{q-d/3-\eps},
\end{eq}
where the estimate holds for all $t \in [0,T]$. Thus, given the assumed form of $D=D(p)$ and provided $\|p\|_{\COmegaT}$ and $\barh$ are sufficiently small, there exist constants $\Dmin$, $\Dmax>0$, such that
\begin{eq}
	0 < \Dmin \leq	D(\ph(\bfx,t)) = D_0(1+D_1 \ph(\bfx,t)) \leq \Dmax, \quad \text{for all  } (\bfx,t) \in \overlineK \times [0,T], K \in \triangles.
\end{eq}
A numerical analysis of the convection-diffusion problem with $D = D(\bfx)$, including sufficient conditions for existence, uniqueness, and optimal convergence, can be found, for example, in~\cite[Chapter 4]{riviere2008discontinuous}.  We now adapt this result to obtain the well-posedness and convergence of the semi-discrete multiphysics problem \eqref{dG semi-discrete wave-adv diff system}. 

\begin{theorem}\label{thm: full model concentration}
	Let the assumptions of Theorem~\ref{thm: wellposedness nonlin} hold with $\qu \geq 1$ if $d \in \{1,2\}$ and $q \geq 2$ if $d=3$.  Let $(p,u) \in \Xp \times \Xu$ be the solution of the exact problem \eqref{exact multiphysics problem},  satisfying the acoustic non-degeneracy condition \eqref{non-degeneracy} and let $\uzeroh= \Ih \uzero$.  Let $\fhu$ and $\ghin$ satisfy the accuracy assumptions \eqref{approx properties source terms}. Then there exist $M>0$ and $\barh>0$, such that if
	\[
	\|p\|_{\COmegaT} \leq M \qquad \text{and} \quad h \leq \barh,
	\]
	then there is a unique solution $(\ph,\uh) \in C^{2}([0,T];\Vhq) \times C^{1}([0,T];\Vhq)$ of \eqref{dG semi-discrete wave-adv diff system}.
	Furthermore, the approximate pressure satisfies the error estimate \eqref{error bound pressure} and the approximate concentration satisfies
	\begin{eq}
		\Ltwonorm{u(t)-\uh(t)}{\Omega}^{2} + \int_{0}^{t} \dgseminorm{u-\uh}^{2} \ds \lesssim h^{2q},
	\end{eq}
for all $t \in [0,T]$, where the hidden constants depend on the exact solution $(p,u)$, the medium parameters, and the final time $T$, but not on the mesh size $h$.
\end{theorem}
\begin{proof}
As Theorem~\ref{thm: wellposedness nonlin} establishes the well-posedness and error estimates for the pressure subproblem, the existence and convergence proofs for the convection-diffusion equation follow by adapting the arguments of~\cite[Thm.~4.2]{riviere2008discontinuous}. We focus on showing the convergence of the approximate concentration. The equations stated in the following should be understood to hold for all $\wh \in \Vhqu$ and for all times $t \in [0,T]$. \\
 \indent Subtracting the weak forms for $u$ and $\uh$ and using the linearity of $\asip(D(\cdot); \cdot, \cdot)$ with respect to its arguments yields
	\begin{eq}\label{eq: subtract u uh}
			&\int_{\Omega} (u_{t}-\uh_{t}) \wh \dx  + \asip(D(\ph);u-\uh,\wh)+ \bupwind(\bfv; u-\uh,\wh) \\
			=&\,\begin{multlined}[t] \int_{\Omega} (f_{u}-\fh_{u})\wh \dx
				- \int_{\Gin}(\uin- \uhin) \wh \, \bfv \cdot \bfn \dG- \asip(D(p);u,\wh) + \asip(D(\ph);u,\wh).
				\end{multlined}
	\end{eq}
Since $D(p) - D(\ph)= \DzeroDone (p-\ph)$,	we can rewrite the $\asip$ terms on the right-hand side as follows:
	\begin{eq}
		&-\asip(D(p);u,\wh) + \asip(D(\ph);u,\wh)  \\
		=&- D_{0}D_{1} \sum_{K \in \triangles} \int_{K} (p-\ph) \nabla u \cdot \nabla \wh \dx +D_{0}D_{1} \sum_{F \in \calFhint}  \intF \jump{\wh}\avg{(p-\ph)}\nabla u\cdot \nF   \dG.
	\end{eq}
	Here  we have also used the fact that $\jump{u}=0$ and $\avg{\nabla u} \cdot \nF = \nabla u \cdot \nF$ for $u \in \Xu$. We include this rewriting in (\ref{eq: subtract u uh}), split the error $u-\uh = u - \Ih u + \Ih u - \uh \eqqcolon \eI + \ehu$, and denote $\epressure = p-\ph$ to arrive at 
	\begin{eq}\label{eq: error equation err concentration}
		&\int_{\Omega} (\eIu_{t} + \ehu_{t}) \wh \dx  + \asip(D(\ph);\eIu + \ehu,\wh)+ \bupwind(\bfv; \eIu + \ehu,\wh) \\
		=&\,\begin{multlined}[t] \int_{\Omega} (\fu-\fh_{u})\wh \dx - \int_{\Gin}(\uin- \uhin) \wh \, \bfv \cdot \bfn \dG
		-\DzeroDone \sum_{K \in \triangles} \int_{K} \epressure \nabla u \cdot \nabla \wh \dx\\ +\DzeroDone \sum_{F \in \calFhint}  \int_{F} \jump{\wh}\avg{\epressure}\nabla u\cdot \nF   \dG.
		\end{multlined}
	\end{eq}
Testing \eqref{eq: error equation err concentration} with $\wh = \ehu$ yields
	\begin{eq} \label{eq: error equation err concentration tested}
		&\frac{1}{2}\timeder \int_{\Omega}  (\ehu)^{2}\dx  + \asip(D(\ph);\ehu,\ehu)+ \bupwind(\bfv; \ehu,\ehu)  \\
		=& \begin{multlined}[t] \int_{\Omega} (f_{u}-\fh_{u})\ehu \dx - \int_{\Gin}(\uin- \uhin) \ehu \, \bfv \cdot \bfn \dG
		-  \Dzero \Done \sum_{K \in \triangles}\intK \epressure \nabla u \cdot \nabla \ehu \dx \\+ \Dzero \Done \sum_{F \in \calFhint}  \intF \jump{\ehu}\avg{\epressure}\nabla u\cdot \nF   \dG 
		- \int_{\Omega} \eI_{t} \ehu \dx - \asip(D(\ph);\eIu,\ehu)\\-\bupwind(\bfv; \eIu ,\ehu).
		\end{multlined}
	\end{eq}
Thanks to Lemma~\ref{lemma: identity bupwind}, we have the identity
\begin{eq} \label{bupwind id}
	\bupwind(\bfv;\ehu,\ehu)  =& \begin{multlined}[t] 	\frac12	\sumFint \intF| \bfvnF| \jump{\ehu}^2 \dG+ \frac12 \sum_{F \subset \Gin} \intF |\bfvnF| (\ehu)^2 \dG \\ +\frac12 \sum_{F \subset \Gout} \intF |\bfvnF| (\ehu)^2 \dG.
	\end{multlined}
\end{eq}
The $\bupwind$ term on the right-hand side of \eqref{eq: error equation err concentration tested} can be bounded using Lemma~\ref{lemma: bound bupwind}:
	\begin{eq} \label{aa}
		|-\bupwind(\bfv; \eIu,\ehu)| \lesssim&\,
		 \begin{multlined}[t]
			\gamma \left\{ \dgseminorm{\ehu}^2 +  \sumFint  \| |\bfvnF|^{1/2} \jump{\ehu}\|^2_{\LtwoF}  +  \sum_{F \subset \Gout}  \| |\bfvnF|^{1/2} \jump{\ehu}\|^2_{\LtwoF}\right\} \\
			+ \|\eIu\|^2_{\Ltwo}+ \sumFint \|(\eIu)_{\textup{up}}\|^2_{L^2(F)} 
			+\sum_{F\subset \Gout}  \|\eIu\|^2_{\LtwoF}
		\end{multlined}
	\end{eq}
	for any $\gamma>0$. The $e^p$ terms on the right-hand side of \eqref{eq: error equation err concentration tested} can be estimated as follows:
	\begin{eq}
&\left|	-  \Dzero \Done \sum_{K \in \triangles}\intK \epressure \nabla u \cdot \nabla \ehu \dx + \Dzero \Done \sum_{F \in \calFhint}  \intF \jump{\ehu}\avg{\epressure}\nabla u\cdot \nF   \dG \right|\\
		\lesssim&\,  \begin{multlined}[t] \sum_{K \in \triangles}  \Ltwonorm{e^{p}}{K} \| \nabla u\|_{L^{\infty}(K)} \Ltwonorm{\nabla \ehu}{K} 
			+  \sum_{F \in \calFhint} \hF \Ltwonorm{\avg{e^{p}}}{F}\|\nabla u\cdot \nF\|_{L^{\infty}(F)} \frac{1}{\hF}\Ltwonorm{\jump{\ehu}}{F}.
		\end{multlined}
	\end{eq}
By the regularity assumption on $u$, the quantities $\|\nabla u\|_{L^{\infty}(K)}$ and $\|\nabla u \cdot \nF\|_{L^{\infty}(F)}$ are uniformly bounded over all $K \in \triangles$ and $F \in \calFhint$. We further use the continuous trace inequality in Lemma~\ref{lem: cont trace ineq} to bound the pressure error on the faces: 
	\begin{equation}
		\sumFint \hF \Ltwonorm{\avg{\epressure}}{F}^{2} \lesssim
		{ \|\nablah e^{p}\|_{L^{2}(\Omega)}^{2}+}
		\Ltwonorm{\epressure}{\Omega}^{2}.
	\end{equation}
	Altogether, we have
		\begin{eq}
		&	\left|	-  \Dzero \Done \sum_{K \in \triangles}\intK \epressure \nabla u \cdot \nabla \ehu \dx + \Dzero \Done \sum_{F \in \calFhint}  \intF \jump{\ehu}\avg{\epressure}\nabla u\cdot \nF   \dG \right|\\
		\lesssim&\,  \gamma \dgseminorm{\ehu}^2+\|\epressure\|^2_{\Ltwo}+\dgseminorm{e^{p}}^{2}. 
	\end{eq}
	Furthermore, we have
	\begin{eq} \label{boundary data bound}
	\left| - \int_{\Gin}(\gin- \ghin) \ehu \, \bfv \cdot \bfn \dG \right| \lesssim \|\gin- \ghin\|^2_{\LtwoGin}+ \gamma \intGin |\bfv \cdot \bfn| (\ehu)^2 \dG.
	\end{eq}
Estimating the $\fu-\fhu$ term by an analogous application of Young's inequality and using coercivity and boundedness of $\asip(D(\cdot); \cdot, \cdot)$ leads to the following estimate:
	\begin{eq}
			&\begin{multlined}[t] \timeder \int_{\Omega}  (\ehu)^{2}\dx +\dgseminorm{\ehu}^{2}\\
				+	\frac12	\sumFint \intF| \bfvnF| \jump{\ehu}^2 \dG+ \frac12 \sum_{F \subset \Gin} \intF |\bfvnF| (\ehu)^2 \dG +\frac12 \sum_{F \subset \Gout} \intF |\bfvnF| (\ehu)^2 \dG
				\end{multlined}\\
		 \lesssim&\, \begin{multlined}[t]
\Ltwonorm{f_{u}-\fh_{u}}{\Omega}^{2} +\|\gin- \ghin\|^2_{\LtwoGin}+  \Ltwonorm{\ehu}{\Omega}^{2}+\|\epressure\|^2_{\Ltwo}+\dgseminorm{e^{p}}^{2}
	\\
		+	\gamma \left\{ \dgseminorm{\ehu}^2 +  \sumFint  \| |\bfvnF|^{1/2} \jump{\ehu}\|^2_{\LtwoF}  +  \sum_{F \subset \Gout}  \| |\bfvnF|^{1/2} \jump{\ehu}\|^2_{\LtwoF}+\intGin |\bfv \cdot \bfn| (\ehu)^2 \dG\right\} \\+ \|\eIu\|^2_{\Ltwo}	+ \Ltwonorm{\eI_{t}}{\Omega}^{2} + \sumFint \|(\eIu)_{\textup{up}}\|^2_{L^2(F)} 
		+\sum_{F\subset \Gout}  \|\eIu\|^2_{\LtwoF}
		\end{multlined}
	\end{eq}
	for any $\gamma >0$.  For a sufficiently small $\gamma>0$, the $\gamma$ terms on the right-hand side can be absorbed by the left-hand side. Then integrating in time, employing the bound \eqref{error bound pressure} on $\epressure$ from Theorem~\ref{thm: wellposedness nonlin} as well as the approximation properties of the interpolant, $\fhu$, and $\ghin$ yields
	\begin{eq}
	\Ltwonorm{\ehu(t)}{\Omega}^{2} + \int_{0}^{t}\dgseminorm{\ehu}^{2} \ds
 \lesssim&\,	\begin{multlined}[t] h^{2q}  + \Ltwonorm{\eh}{0,t;L^{2}(\Omega)}^{2} + \Ltwonorm{\epressure}{0,T;L^{2}(\Omega)}^{2} 
		+\int_{0}^{T}\dgseminorm{e^{p}}^{2}\ds 
		\end{multlined}
		\\
		\lesssim&\, h^{2q} + \Ltwonorm{\eh}{0,t;L^{2}(\Omega)}^{2}.
	\end{eq}
	The result then follows by applying \Gronwall's inequality. 
\end{proof}
\subsection{Bounds on the approximate concentration}
Since $u$ models a concentration, we expect that $\uh \in [0,1]$ in $\overline{\Omega} \times [0,T]$. If $\uh < 0$, it is said that an \textit{undershoot} occurs, and if $\uh>1$, we have \textit{overshoot}. These phenomena depend on both spatial and time discretization. 
In the semidiscrete solution, the overshoot can be bounded in terms of the mesh size $h$ and the exact solution. We have
\begin{eq}
	\| \uh \|_{L^{\infty}(\Omega)} \leq&\, \| u \|_{L^{\infty}(\Omega)} + \| \Ih u -\uh \|_{L^{\infty}(\Omega)} + \| \Ih u -{u} \|_{L^{\infty}(\Omega)}\\
	 \leq&\, \| u \|_{L^{\infty}(\Omega)} + C \left( h^{q-d/2}+ h^{q+1-d/2} \right)
\end{eq}
for all $t \in [0,T]$, where we have used Theorem~\ref{thm: full model concentration} in the last step. For $q \geq 2$, we can guarantee that $\uh(t) < 1$ provided $h$ is small relative to $1- \|u(t)\|_{\Linf}$. \\
\indent The undershoot is a more difficult phenomenon. In general, we cannot guarantee positivity for the considered dG method. Convection-dominated settings are, in particular, prone to undershoot.  A popular way to reduce undershoot is by using \textit{slope limiters}; see~\cite[Section 4.3.2]{riviere2008discontinuous}. In specific cases, one can design a positivity-preserving method; we refer, e.g., to~\cite{pospresandvp} in which the compressible Euler equations are discretized using a dG method in space and a Lax--Wendroff method in time. For such methods, slope limiting is not needed.
%

\subsection{Numerical results for the pressure-dependent concentration}

To verify the theoretical result in Theorem~\ref{thm: full model concentration}, we next perform an empirical order of convergence study following the same academic setting as in Section~\ref{sec: pressure numerics} and then use a more realistic setting to illustrate the effects of ultrasound-enhanced drug transport numerically. 

\subsubsection{Empirical order of convergence for the concentration}\label{sec: error numerics concentration}

We set $\barOmega = [0,1] \times [0,2]$ again as in Section~\ref{sec: pressure numerics}. We choose the source terms and boundary and initial data such that the exact solution for the concentration is
\begin{eq}
	u(x,y,t) = e^{-t}\cos(\pi y),
\end{eq}
and the exact pressure is again given by \eqref{exact pressure}. We set $D_{0} = D_{1} =1$ in the diffusion coefficient, so that $D(p)= 1+p$.
We take $\bfv=(0,1)^{\textup{T}}$, so that $\Gin = \{(x,y) \in \partial \Omega: \, y=0\}$. The Newmark scheme with a very small time step  is again used to solve the Westervelt equation, and we employ the backward Euler method for the time stepping for concentration. At each time step, the pressure is first computed and then used to advance the concentration. The resulting empirical orders of convergence are reported in Figure~\ref{fig: coupled}.

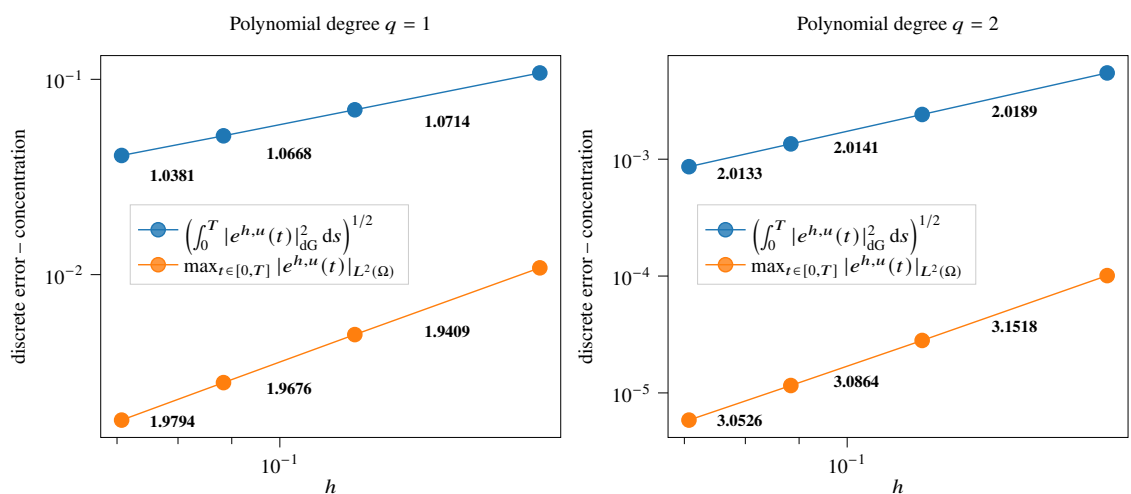
\begin{figure}[h]
	\centering
	\begin{subfigure}[b]{0.49\textwidth}
		\centering
		\resizebox{\textwidth}{!}{
\begin{tikzpicture}

\definecolor{darkgray176}{RGB}{176,176,176}
\definecolor{darkorange25512714}{RGB}{255,127,14}
\definecolor{lightgray204}{RGB}{204,204,204}
\definecolor{steelblue31119180}{RGB}{31,119,180}

\begin{axis}[
legend cell align={left},
legend style={
  fill opacity=0.8,
  draw opacity=1,
  text opacity=1,
  at={(0.67,0.4)},
  anchor=south east,
  draw=lightgray204
},
log basis x={10},
log basis y={10},
minor xtick={0.002,0.003,0.004,0.005,0.006,0.007,0.008,0.009,0.02,0.03,0.04,0.05,0.06,0.07,0.08,0.09,0.2,0.3,0.4,0.5,0.6,0.7,0.8,0.9,2,3,4,5,6,7,8,9,20,30,40,50,60,70,80,90},
%
tick align=outside,
tick pos=left,
title={Polynomial degree $q = 1$},
x grid style={darkgray176},
xlabel={\(\displaystyle h\)},
xmin=0.0675441905923626, xmax=0.185064028310577,
xmode=log,
xtick style={color=black},
xtick={0.001,0.01,0.1,1,10},
xticklabels={
  \(\displaystyle {10^{-3}}\),
  \(\displaystyle {10^{-2}}\),
  \(\displaystyle {10^{-1}}\),
  \(\displaystyle {10^{0}}\),
  \(\displaystyle {10^{1}}\)
},
y grid style={darkgray176},
ylabel={discrete error -- concentration},
ymin=0.00146706608060219, ymax=0.132290271392981,
ymode=log,
ytick style={color=black},
ytick={0.0001,0.001,0.01,0.1,1,10},
yticklabels={
  \(\displaystyle {10^{-4}}\),
  \(\displaystyle {10^{-3}}\),
  \(\displaystyle {10^{-2}}\),
  \(\displaystyle {10^{-1}}\),
  \(\displaystyle {10^{0}}\),
  \(\displaystyle {10^{1}}\)
}
]
\addplot [semithick, steelblue31119180, mark=*, mark size=3, mark options={solid}]
table {%
0.176776695296637 0.107810414325238
0.117851130197758 0.0698238438336352
0.0883883476483184 0.0513718519737647
0.0707106781186548 0.040749134457278
};
\addlegendentry{$\left(\int_0^T |e^{h,u}(t)|^2_{\dGnorm} \ds \right)^{1/2}$}
\addplot [semithick, darkorange25512714, mark=*, mark size=3, mark options={solid}]
table {%
0.176776695296637 0.0108328445511759
0.117851130197758 0.00493145558758793
0.0883883476483184 0.00279991186388867
0.0707106781186548 0.00180018387990618
};
\addlegendentry{$\max_{t \in [0,T]} |e^{h,u}(t)|_{L^2(\Omega)}$}
\draw (axis cs:0.144337567297406,0.060733774717266) node[
  scale=0.5,
  text=black,
  rotate=0.0
]{\Large \bf 1.0714};
\draw (axis cs:0.102062072615966,0.0419239821955617) node[
  scale=0.5,
  text=black,
  rotate=0.0
]{\Large \bf 1.0668};
\draw (axis cs:0.0790569415042095,0.0320272644268157) node[
  scale=0.5,
  text=black,
  rotate=0.0
]{\Large \bf 1.0381};
\draw (axis cs:0.144337567297406,0.00511631009398085) node[
  scale=0.5,
  text=black,
  rotate=0.0
]{\Large \bf 1.9409};
\draw (axis cs:0.102062072615966,0.00260110439869385) node[
  scale=0.5,
  text=black,
  rotate=0.0
]{\Large \bf 1.9676};
\draw (axis cs:0.0790569415042095,0.00177155163429011) node[
  scale=0.5,
  text=black,
  rotate=0.0
]{\Large \bf 1.9794};
\end{axis}

\end{tikzpicture}}
	\end{subfigure}
	\begin{subfigure}[b]{0.49\textwidth}
		\centering
		\resizebox{\textwidth}{!}{
\begin{tikzpicture}

\definecolor{darkgray176}{RGB}{176,176,176}
\definecolor{darkorange25512714}{RGB}{255,127,14}
\definecolor{lightgray204}{RGB}{204,204,204}
\definecolor{steelblue31119180}{RGB}{31,119,180}

\begin{axis}[
legend cell align={left},
legend style={
  fill opacity=0.8,
  draw opacity=1,
  text opacity=1,
  at={(0.67,0.4)},
  anchor=south east,
  draw=lightgray204
},
log basis x={10},
log basis y={10},
minor xtick={0.002,0.003,0.004,0.005,0.006,0.007,0.008,0.009,0.02,0.03,0.04,0.05,0.06,0.07,0.08,0.09,0.2,0.3,0.4,0.5,0.6,0.7,0.8,0.9,2,3,4,5,6,7,8,9,20,30,40,50,60,70,80,90},
tick align=outside,
tick pos=left,
title={Polynomial degree $q = 2$},
x grid style={darkgray176},
xlabel={\(\displaystyle h\)},
xmin=0.0675441905923626, xmax=0.185064028310577,
xmode=log,
xtick style={color=black},
xtick={0.001,0.01,0.1,1,10},
xticklabels={
  \(\displaystyle {10^{-3}}\),
  \(\displaystyle {10^{-2}}\),
  \(\displaystyle {10^{-1}}\),
  \(\displaystyle {10^{0}}\),
  \(\displaystyle {10^{1}}\)
},
y grid style={darkgray176},
ylabel={discrete error -- concentration},
ymin=4.15793172325948e-06, ymax=0.00770399387885243,
ymode=log,
ytick style={color=black},
ytick={1e-07,1e-06,1e-05,0.0001,0.001,0.01,0.1},
yticklabels={
  \(\displaystyle {10^{-7}}\),
  \(\displaystyle {10^{-6}}\),
  \(\displaystyle {10^{-5}}\),
  \(\displaystyle {10^{-4}}\),
  \(\displaystyle {10^{-3}}\),
  \(\displaystyle {10^{-2}}\),
  \(\displaystyle {10^{-1}}\)
}
]
\addplot [semithick, steelblue31119180, mark=*, mark size=3, mark options={solid}]
table {%
0.176776695296637 0.00547239971975762
0.117851130197758 0.00241362302579455
0.0883883476483184 0.00135216284457566
0.0707106781186548 0.000862815105470303
};
\addlegendentry{$\left(\int_0^T |e^{h,u}(t)|^2_{\dGnorm} \ds \right)^{1/2}$}
\addplot [semithick, darkorange25512714, mark=*, mark size=3, mark options={solid}]
table {%
0.176776695296637 0.000100892901400043
0.117851130197758 2.81091455672475e-05
0.0883883476483184 1.15674745828623e-05
0.0707106781186548 5.85349795063876e-06
};
\addlegendentry{$\max_{t \in [0,T]} |e^{h,u}(t)|_{L^2(\Omega)}$}
\draw (axis cs:0.144337567297406,0.00254402670687234) node[
  scale=0.5,
  text=black,
  rotate=0.0
]{\Large \bf 2.0189};
\draw (axis cs:0.102062072615966,0.00126458276691679) node[
  scale=0.5,
  text=black,
  rotate=0.0
]{\Large \bf 2.0141};
\draw (axis cs:0.0790569415042095,0.000756086369672296) node[
  scale=0.5,
  text=black,
  rotate=0.0
]{\Large \bf 2.0133};
\draw (axis cs:0.144337567297406,3.72779625724947e-05) node[
  scale=0.5,
  text=black,
  rotate=0.0
]{\Large \bf 3.1518};
\draw (axis cs:0.102062072615966,1.26223767642471e-05) node[
  scale=0.5,
  text=black,
  rotate=0.0
]{\Large \bf 3.0864};
\draw (axis cs:0.0790569415042095,5.76003407062639e-06) node[
  scale=0.5,
  text=black,
  rotate=0.0
]{\Large \bf 3.0526};
\end{axis}

\end{tikzpicture}}
	\end{subfigure}
	\caption{Discrete errors for the concentration}
	\label{fig: coupled}
\end{figure}

We see that the method indeed converges as expected in the dG-based semi-norms. The error measured in the $L^\infty(0,T; \Ltwo)$ norm  converges at the higher rate $h^{q+1}$, exceeding the theoretical predictions of Theorem~\ref{thm: full model concentration}, as is typical for dG methods.

\subsubsection{A further numerical study}\label{sec: simulations concentration}

Finally, we consider a setting with physically realistic acoustic parameters. We take $\barOmega = [0, 0.01\si{\meter}]^2$ and set $T =  5 \cdot 10^{-6}\si{\second}$, $\Delta t = 5 \cdot 10^{-8}\si{\second}$, and take $N_{x}=N_{y}= 40$ elements in both directions.
We again use the Newmark scheme for acoustic time stepping with parameters $(0.25, 0.5)$. The acoustic medium parameters are taken to be $c = 1500\si[per-mode=symbol]{\meter \per \second}$, $\beta = 10^{-6} \si[per-mode=symbol]{\meter \squared \per\second}$, and $\kappa = 1 \si{\per \pascal}$. We set $\alpha = c$ and $\gA=0$, corresponding to zero-order Engquist--Majda absorbing conditions. The acoustic source term is given by
\begin{eq}
	\fp(x,y, t) = {3 \cdot}10^{11} e^{-\frac{1}{\sigma_0^{2}}((x- 0.005)^{2}+(y- 0.005)^{2})}\sin(2\pi \omega t),
\end{eq}
with $\sigma_0 = 2 \cdot 10^{-4}\si{\meter}$ and frequency $\omega = 400\, \si{k\Hz}$. \\
\indent For the concentration, the drug enters at the bottom side $\Gin = \{(x,y) \in \partial \Omega: y=0\}$ through a pulse $\gin = 1 \si[per-mode=symbol]{\kilo \gram \per \meter \cubed}$. 
The convective velocity is therefore directed in the upward $y$-direction: $\bfv = (0,10^{-3}\si[per-mode=symbol]{\meter \per \second})^{\textup{T}}$. \\
\indent  
The drug concentration evolves on a much slower time scale than the acoustic pressure. In order to capture the effect of the ultrasound, we set
\[
D(p)=D_{0}(1+D_{1}|p|),
\] 
where the absolute value ensures that both positive and negative pressure values contribute to increasing the diffusion coefficient, and, following~\cite{careaga2025westervelt}, take a relatively large value $\Dzero = 5 \si[per-mode=symbol]{\meter \squared \per \second}$.  This diffusion-dominated regime also helps suppress undershoot artifacts (see~\cite[Chapter 4]{riviere2008discontinuous}). We further set $\Done = 500 \si{\per  \pascal}$.  Figures~\ref{fig: pressure waves} and~\ref{fig: concentration spread} show the propagation of acoustic waves and the spread of the drug concentration, respectively, at four different times. 
\begin{figure}[h]
	\centering
	\includegraphics[width=0.7\linewidth]{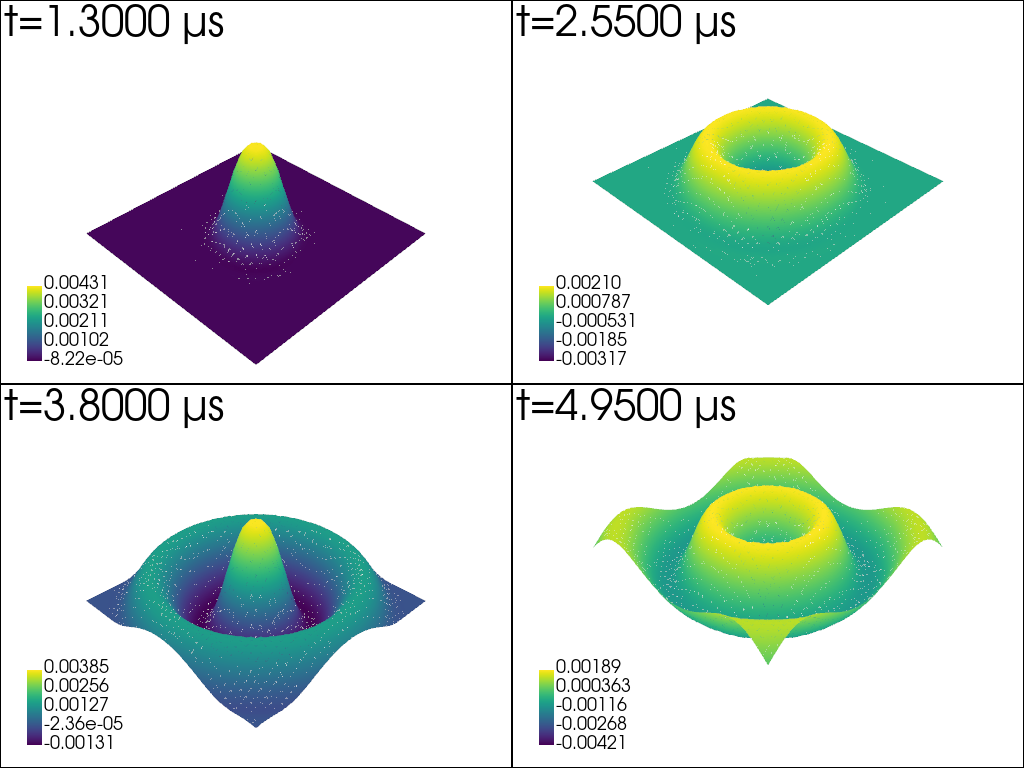}
	\caption{Pressure waves affecting drug spread ($[\si[per-mode=symbol]{\pascal}]$).}
	\label{fig: pressure waves}
\end{figure}
\begin{figure}[h]
	\centering
	\includegraphics[width=0.7\linewidth]{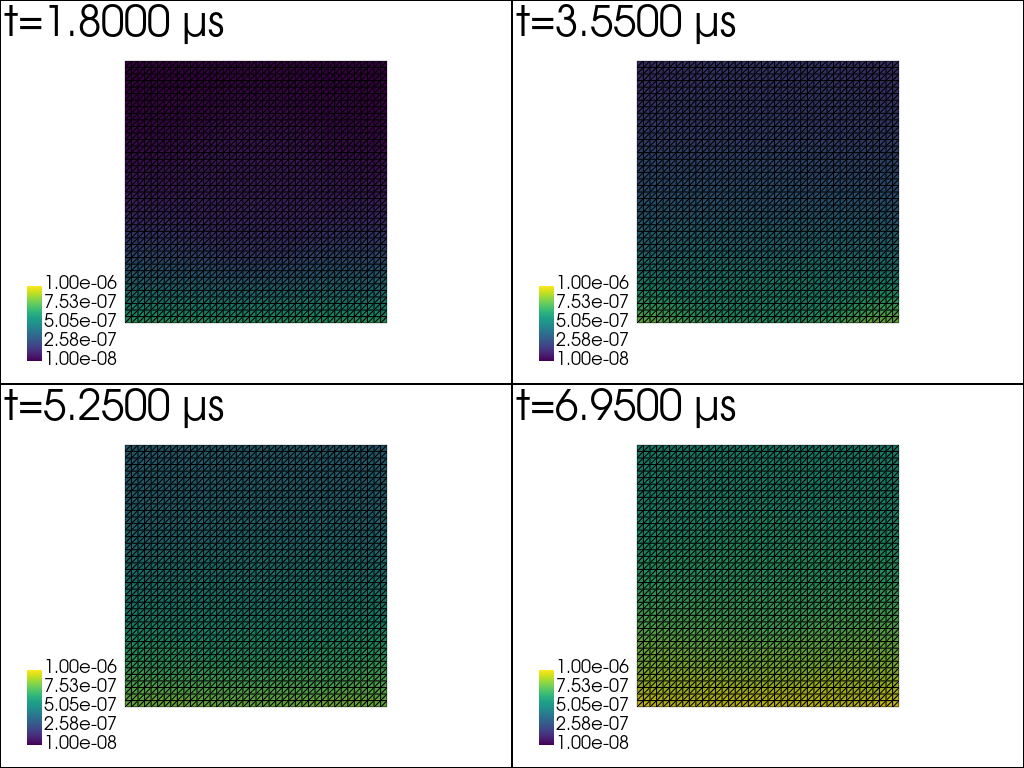}
	\caption{Ultrasound-enhanced drug concentration at different times ($[\si[per-mode=symbol]{\kilo \gram \per \meter \cubed}]$).}
	\label{fig: concentration spread}
\end{figure}
\begin{figure}[h]
	\captionsetup{width=.65\linewidth}
	\centering
\begin{tikzpicture}

\definecolor{darkgray176}{RGB}{176,176,176}
\definecolor{steelblue31119180}{RGB}{31,119,180}

\begin{axis}[
	tick align=outside,
	width=.55\textwidth,
	tick pos=left,
	title={Relative change in concentration at $y=0.01$},
	xlabel={$t$},
	ylabel={$\delta_{\Gamma_4}$},
	x grid style={darkgray176},
	xmin=-2.5e-07, xmax=5.25e-06,
	xtick style={color=black},
	xtick={-1e-06,0,1e-06,2e-06,3e-06,4e-06,5e-06,6e-06},
	xticklabels={\ensuremath{-}1,0,1,2,3,4,5,6},
	y grid style={darkgray176},
	ymin=-0.0175976023605942, ymax=0.369549649572478,
	ytick style={color=black}
	]
\addplot [thick, steelblue31119180]
table {%
0 0
5e-08 2.61229980063544e-13
1e-07 1.09735830507051e-11
1.5e-07 1.54485530705206e-10
2e-07 1.22456395733605e-09
2.5e-07 6.68935697169072e-09
3e-07 2.8046024051904e-08
3.5e-07 9.63183489938757e-08
4e-07 2.82809657953884e-07
4.5e-07 7.31877878157609e-07
5e-07 1.7072669210839e-06
5.5e-07 3.65167830405538e-06
6e-07 7.25808111114197e-06
6.5e-07 1.35491881301274e-05
7e-07 2.39618411300025e-05
7.5e-07 4.0430568428519e-05
8e-07 6.54675722410322e-05
8.5e-07 0.000102235432677928
9e-07 0.00015460521173165
9.5e-07 0.000227205205330838
1e-06 0.000325452016722789
1.05e-06 0.000455562756977896
1.1e-06 0.000624548624246397
1.15e-06 0.000840194516487691
1.2e-06 0.00111101598027107
1.25e-06 0.00144617502296986
1.3e-06 0.00185541126286028
1.35e-06 0.00234894083127356
1.4e-06 0.00293730026501107
1.45e-06 0.003631178267044
1.5e-06 0.00444128939887773
1.55e-06 0.00537815516971969
1.6e-06 0.00645173288941508
1.65e-06 0.007671343188988
1.7e-06 0.0090455197029384
1.75e-06 0.0105816459252484
1.8e-06 0.0122863803799184
1.85e-06 0.0141662860561016
1.9e-06 0.0162275666895501
1.95e-06 0.0184758653322891
2e-06 0.0209174572574256
2.05e-06 0.0235586660785448
2.1e-06 0.0264045122492419
2.15e-06 0.0294608250595325
2.2e-06 0.0327334291756581
2.25e-06 0.0362263161805332
2.3e-06 0.0399433496232247
2.35e-06 0.0438892743534582
2.4e-06 0.048066753216345
2.45e-06 0.0524755945456339
2.5e-06 0.057119316246095
2.55e-06 0.0619978612945942
2.6e-06 0.0671062513725169
2.65e-06 0.072442890973863
2.7e-06 0.0780043790988953
2.75e-06 0.0837794913316373
2.8e-06 0.0897551311357059
2.85e-06 0.0959241217723828
2.9e-06 0.102267802695381
2.95e-06 0.108762332420165
3e-06 0.115396633430663
3.05e-06 0.122152423647481
3.1e-06 0.129004004193183
3.15e-06 0.135936042612656
3.2e-06 0.142937062885125
3.25e-06 0.149987470817112
3.3e-06 0.157069170150075
3.35e-06 0.164171402499201
3.4e-06 0.17128308870612
3.45e-06 0.178390849841438
3.5e-06 0.185484217888243
3.55e-06 0.192559908810775
3.6e-06 0.199613094668033
3.65e-06 0.206631528658771
3.7e-06 0.213617036281972
3.75e-06 0.220565750003152
3.8e-06 0.227464370563231
3.85e-06 0.234305738182875
3.9e-06 0.241090775506928
3.95e-06 0.247796338120658
4e-06 0.254410919419225
4.05e-06 0.260929745293588
4.1e-06 0.267336651102474
4.15e-06 0.273594618696164
4.2e-06 0.27970924450103
4.25e-06 0.285664027996141
4.3e-06 0.29142669956948
4.35e-06 0.29699286379047
4.4e-06 0.30236593037116
4.45e-06 0.30753018846126
4.5e-06 0.312461533019186
4.55e-06 0.317202724550899
4.6e-06 0.321743200732056
4.65e-06 0.32607055134619
4.7e-06 0.330208693457945
4.75e-06 0.334190898001139
4.8e-06 0.338010382135746
4.85e-06 0.341662719686733
4.9e-06 0.345203934998139
4.95e-06 0.348627989258819
5e-06 0.351952047211884
};
\end{axis}

\end{tikzpicture}
	\caption{Relative change \eqref{rel change} between ultrasound-enhanced concentration and the concentration computed with a constant diffusion coefficient at the top boundary $y=0.01$.}
	\label{fig: concentration difference}
\end{figure}
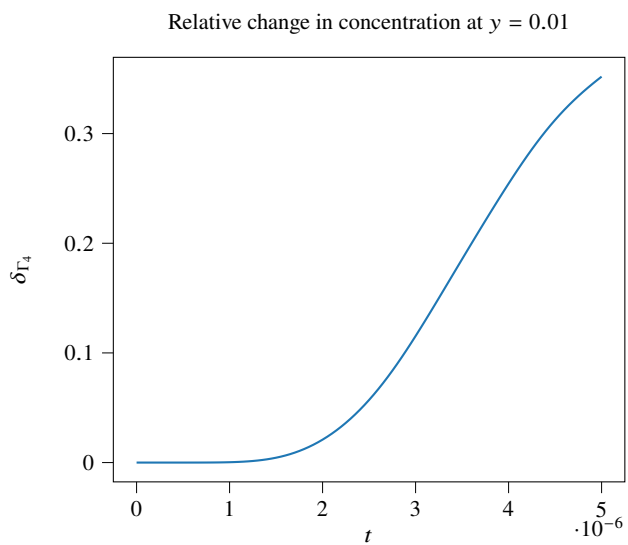
\newpage
Figure \ref{fig: concentration difference} shows the relative average difference between the ultrasound-enhanced concentration $u$ and the concentration $u_{D=D_0}$ computed using constant diffusion coefficient $D=D_0$
at the top boundary $\Gtop= \{(x,y) \in \partial \Omega: y= 0.01 \}$:
\begin{eq} \label{rel change}
	\delta_{\Gtop}(t) =	\frac{\int_{\Gtop} u(t) \dG - \int_{\Gtop} u_{D=D_0}(t) \dG}{\displaystyle \max_{[0,T]} \int_{\Gout} u_{D=D_0}(t) \dG}.
\end{eq}
We observe that the maximal relative change at the top boundary is around $35 \%$. 
\section*{Conclusion and outlook}\label{sec: conclusion}
In this work, we have analyzed a discontinuous Galerkin semi-discretization of a wave--convection-diffusion system that captures the influence of ultrasound waves on the diffusivity of drug concentration. We have established optimal \emph{a priori} error estimates in energy norms under suitable assumptions on the exact solution, discretization parameter, and polynomial degree.  An open theoretical question is to establish the optimal order of convergence in the $L^\infty(0,T; L^{2}(\Omega))$ norm. This could also lead to weaker assumptions on the polynomial degree in Theorem~\ref{thm: wellposedness nonlin} as one could exploit the estimate $\|\kappa (\ph - \Ih p)(\thh) \|_{L^{\infty}(\Omega)} \lesssim h^{-d/2} \Ltwonorm{\kappa (\ph - \Ih p)(\thh)}{\Omega} \lesssim h^{q+1-d/2}$.  \\
\indent A natural further extension of the theory concerns the convective velocity. In a more complete physical model,  the convective velocity $\bfv$ is also expected to depend on the ultrasound pressure; see, for example,~\cite{ferreira2022drug} for modeling details. Discussions on the mathematical treatment of a space-variable advection field can be found, for example, in~\cite[Sec.~5]{CangianDongGeorgoulisHouston2017} and~\cite{houston2002discontinuous}. Future work could investigate conditions under which the present methodology can be adapted to accommodate a pressure-dependent velocity.
\begin{acknowledgement}
The work of F.d.W. was partially supported by research grant no. G066725N from the Research Foundation - Flanders (FWO).  
The work of V.N.  was partially supported by the Dutch Research Council (NWO) under the grant OCENW.M.23.371 \sloppy with Grant ID \href{https://doi.org/10.61686/VLHHB85047 }{https://doi.org/10.61686/VLHHB85047}. 
\end{acknowledgement}
\ethics{Competing Interests}{
The authors have no conflicts of interest to declare that are relevant to the content of this chapter.}


\bibliographystyle{abbrv} 
\bibliography{references}
\end{document}